\documentclass{siamltex}

\usepackage{overpic}
\usepackage{amsmath,amstext,amsbsy,amssymb,subfig,mathdots}
\usepackage{booktabs}
\usepackage{mathrsfs}
\usepackage{tikz}
\usetikzlibrary{positioning}
\usetikzlibrary{shapes,arrows}

\newenvironment{remark}{{\bf Remark}}

\newtheorem{example}{Example}

\def\B{{\mathfrak B}}
\def\L{{\cal L}}
\def\O{{\cal O}}

\def\I{{\rm i}}

\def\Bx{\B_x}
\def\By{\B_y}

\def\abs#1{\left|{#1}\right|}

\def\dkf{{\mathrm d}}

\let\oldremark\remark\renewcommand{\remark}{\oldremark\normalfont}
\let\oldexample\example\renewcommand{\example}{\oldexample\normalfont}

\newcount\examplenumber
\def\Example#1{\advance \examplenumber by 1
\subsection{Example \the\examplenumber\ (#1)}}

\def\Ki_#1^#2{{\mathscr K}_{#1}^{#2}}

\def\sopmatrix#1{\begin{pmatrix} #1 \end{pmatrix}}

\def\figref#1{Figure~\ref{fig:#1}}

\def\addtab#1={#1\;&=}

\def\meeq#1{\def\ccr{\\\addtab}
 \begin{align*}
 \addtab#1
 \end{align*}
  }

\def\vc#1{\mbox{\boldmath$#1$\unboldmath}}

\def\Figure#1#2\par{
\begin{figure}[tb]
\begin{center}{
\includegraphics{Figures/#1}}
\end{center}
\caption{#2}\label{fig:#1} 
\end{figure}
}
\def\Figurew#1#2#3\par{
\begin{figure}[tb]
\begin{center}{
\includegraphics[width=#2]{Figures/#1}}
\end{center}
\caption{#3}\label{fig:#1} 
\end{figure}
}

\def\Figuretwow#1#2#3#4\par{
\begin{figure}[tb]
\begin{center}{
\includegraphics[width=#3]{Figures/#1}\includegraphics[width=#3]{Figures/#2}}
\end{center}
\caption{#4}\label{fig:#1} 
\end{figure}
}
\def\Figuretwo#1#2#3\par{
	\Figuretwow{#1}{#2}{0.48 \hsize}
		#3\par	
}

\def\pr(#1){\left({#1}\right)}
\def\br[#1]{\left[{#1}\right]}
\def\set#1{\left\{{#1}\right\}}
\def\ip<#1>{\left\langle{#1}\right\rangle}
\def\iip<#1>{\left\langle\!\langle{#1}\right\rangle\!\rangle}

\def\fpr(#1){\!\pr({#1})}

\def\function#1#2{\expandafter\def\csname #1\endcsname(##1){#2\fpr({##1})}
				\expandafter\def\csname #1p\endcsname(##1){#2'\fpr({##1})}
				\expandafter\def\csname #1pp\endcsname(##1){#2''\fpr({##1})}
				\expandafter\def\csname #1pn\endcsname##1(##2){#2^{\pr({##1})}\fpr({##2})}				}
\def\defoperator#1#2{\expandafter\def\csname #1\endcsname[##1]{#2\!\br[{##1}]}}

\def\ffunction#1{\function{#1}{#1}}

\def\O(#1){{\cal O}\!\left(#1\right)}

\def\Oo(#1){{\rm o}\!\left({#1}\right)}

\def\fO(#1){\Oh\fpr({#1})}

\def\dkfn^#1#2{\,{\rm d}^#1 #2}
\def\dkf#1{\,{\rm d}#1}

\def\I{{\rm i}}

\def\H_#1^#2(#3){H_#1^{(#2)}\fpr({#3})}
\def\J_#1(#2){{\rm J}_#1\!\pr(#2)}

\def\ddxn^#1{{\dkf{}^#1 \over \dkfn^#1{x}}}

\def\seq_#1^#2#3{\set{#3_{#1},\ldots,#3_{#2}}}

\def\abs#1{\left|{#1}\right|}

\ffunction{f}
\ffunction{g}

\def\mapengine#1,#2.{\mapfunction{#1}\ifx\void#2\else\mapengine #2.\fi }

\def\map[#1]{\mapengine #1,\void.}

\def\mapenginesep_#1#2,#3.{\mapfunction{#2}\ifx\void#3\else#1\mapengine #3.\fi }

\def\mapsep_#1[#2]{\mapenginesep_{#1}#2,\void.}

\def\vcbr[#1]{\pr({#1})}

\def\bvect[#1,#2]{
{
\def\dots{\cdots}
\def\mapfunction##1{\ | \  ##1}
	\sopmatrix{
		 \,#1\map[#2]\,
	}
}
}

\def\vect[#1]{
{\def\dots{\ldots}
	\vcbr[{#1}]
}}

\def\vectt[#1]{
{\def\dots{\ldots}
	\vect[{#1}]^{\top}
}}

\def\Vectt[#1]{
{
\def\mapfunction##1{##1 \cr} 
\def\dots{\vdots}
	\sopmatrix{
		\map[#1]
	}
}}

\def\simlimit_#1{\,\,\,\sim \!\!\!\!\!\!\!\!\!{ \atop \scriptscriptstyle #1 }}

\def\XXint#1#2#3{{\setbox0=\hbox{$#1{#2#3}{\int}$}
     \vcenter{\hbox{$#2#3$}}\kern-.5\wd0}}

\def\rad^#1#2{\,\,{}^#1\!\!\!\!\sqrt{#2}\,}

\def\Figuretwofixed#1#2#3\par{
\Figuretwow{#1}{#2}{0.48 \hsize}{#3}\par
}

\title{The automatic solution of partial differential equations using a global spectral method}

\author{Alex Townsend\thanks{Department of Mathematics, Massachusetts Institute of Technology, 77 Massachusetts Avenue
Cambridge, MA 02139-4307. (ajt@mit.edu)} \and 
Sheehan Olver\thanks{School of Mathematics and Statistics, The University of Sydney, Sydney, Australia. (Sheehan.Olver@sydney.edu.au)} 
}

\begin{document}
\maketitle

\begin{abstract}
A spectral method for solving linear partial differential equations (PDEs) with variable coefficients and general boundary conditions defined on rectangular domains is described, 
based on separable representations of partial differential operators and the one-dimensional ultraspherical spectral method. 
If a partial differential operator is of splitting rank~$2$, such as the operator associated with Poisson or Helmholtz, 
the corresponding PDE is solved via a generalized Sylvester matrix equation, and a bivariate polynomial 
approximation of the solution of degree $(n_x,n_y)$ is computed in $\mathcal{O}((n_x n_y)^{3/2})$ operations. Partial 
differential operators of splitting rank~$\geq 3$ are solved via a linear system involving a block-banded matrix in 
$\mathcal{O}(\min(n_x^{3} n_y,n_x n_y^{3}))$ operations. 
Numerical examples demonstrate the applicability of our 2D spectral method to a broad class of PDEs, which includes elliptic and dispersive 
time-evolution equations.    
The resulting PDE solver is written in {\sc Matlab} and is publicly available as part of {\sc Chebfun}. It can resolve solutions requiring over a million degrees of freedom 
in under $60$ seconds. An experimental implementation in the Julia language can currently perform the same solve in $10$ seconds.
\end{abstract}

\begin{keywords}
Chebyshev, ultraspherical, partial differential equation, spectral method
\end{keywords}

\begin{AMS}
33A65, 35C11, 65N35
\end{AMS}

\section{Introduction}
This paper describes a spectral method for the solution of linear partial 
differential equations (PDEs) with variable coefficients defined on bounded rectangular domains $[a,b]\times[c,d]$ that take the form: 
\begin{equation}\label{eq:PDE}
\L u(x,y) = f(x,y), \qquad \mathcal{L} = \sum_{\vphantom{j=0}i=0}^{N_y} \sum_{j=0}^{\vphantom{N_y}N_x} \ell_{ij}(x,y)\frac{\partial^{i+j}}{\partial y^i \partial x^j},
\end{equation}
where $N_x$ and $N_y$ are the differential orders of $\mathcal{L}$ in the $x$- and $y$-variable, respectively, $f(x,y)$ and $\ell_{ij}(x,y)$ are 
functions defined on $[a,b]\times[c,d]$, and $u(x,y)$ is the desired solution. The operator $\mathcal{L}$ is called 
a linear partial differential operator (PDO). Many real-world 
phenomena can be formalized in terms of a PDE; see, for example,~\cite{Evan_10_01,Farlow_93_01,Fritz_82_01}.

In addition,~\eqref{eq:PDE} should be 
supplied with $K_x,K_y\geq 0$ linear constraints, i.e.,
\[
\mathfrak{B}_x u(x,y) =\vc g(y), \qquad
\mathfrak{B}_y u(x,y) = \vc h(x),
\]
to ensure that there is a unique solution. Here, $\vc g$ and $\vc h$ are vector-valued functions with $K_x$ and $K_y$ components
and $\Bx$ and  $\By$ are linear operators acting on continuous bivariate functions, which usually, but not necessarily, represent boundary 
conditions on the left-right and top-bottom edges of $[a,b]\times [c,d]$. For example, if $\mathfrak{B}_x$ and  $\mathfrak{B}_y$ represent Dirichlet boundary conditions, then $K_x = K_y=2$,
\[
\mathfrak{B}_x u(x,y) = \begin{pmatrix}u(a,y)\cr  u(b,y)\end{pmatrix},\qquad \mathfrak{B}_y u(x,y) = \begin{pmatrix}u(x,c)\cr  u(x,d)\end{pmatrix},
\]
and $\vc g$ and $\vc h$ are the prescribed boundary data along the four edges. The spectral method we describe allows for general linear constraints such as  
Neumann and Robin boundary conditions, as well as the possibility of interior and integral constraints. Without loss of generality, the 
constraints are assumed to be linearly independent; otherwise, at least one of them can be removed while preserving the uniqueness of the solution.
In this paper we always assume there is a unique solution and seek an accurate numerical approximation to it. 

For integers $n_x$ and $n_y$ our spectral method returns a matrix $X\in\mathbb{C}^{n_y\times n_x}$ of bivariate Chebyshev expansion coefficients 
for the solution $u$~\cite[Sect.~2(c)]{Basu_73_01} such that  
\begin{equation}
 u(x,y) \approx \sum_{\vphantom{j=0}i=0}^{n_y-1} \sum_{j=0}^{\vphantom{n_y-1}n_x-1} X_{ij} T_i(\psi(y)) T_j(\phi(x)), \qquad (x,y)\in[a,b]\times [c,d],
\label{eq:bivariateChebyshevExpansion}
\end{equation} 
where $T_j(x) = \cos(j\cos^{-1}x)$ for $x\in[-1,1]$ is the degree $j$ Chebyshev polynomial (of the first kind), and $\phi(x) = 2(x-a)/(b-a)-1$ and $\psi(y) = 2(y-c)/(d-c) - 1$ 
are affine transformations from $[a,b]$ and $[c,d]$ to $[-1,1]$, respectively. The approximant in~\eqref{eq:bivariateChebyshevExpansion} 
is of degree $(n_x-1,n_y-1)$, i.e., of degree $n_x-1$ in $x$ and $n_y-1$ in $y$. In practice, we adaptively determine $n_x$ and $n_y$ so that the computed bivariate 
polynomial in~\eqref{eq:bivariateChebyshevExpansion} uniformly approximates the solution on $[a,b]\times [c,d]$ to a high accuracy (see Section~\ref{sec:discretization}).   

There are many exemplary papers that focus on solving a specific PDE and developing specialized algorithms to do so, for example,~\cite{Buzbee_70_01,Haldenwang_84_01}. 
In contrast, we concentrate on what can be achieved by a general solver that is merely given a description of a PDE in a syntax close to the notation found in 
standard textbooks~\cite{Evan_10_01,Farlow_93_01,Fritz_82_01}. This opens up a wonderful opportunity for a computational scientist to creatively explore and 
investigate in a way that can be very fruitful.  As an example the following {\sc Matlab} code solves the Helmholtz 
equation given by $u_{xx} + u_{yy} + 1000u = \cos(10xy)$ on $[-1,1]\times [-1,1]$ with non-homogeneous Dirichlet conditions:
\begin{verbatim}
 N = chebop2(@(u) laplacian(u) + 1000*u);    % N = u_xx+u_yy+1000u
 N.lbc = 1; N.rbc = 1; N.dbc = 1; N.ubc = 1; % u=1 at boundary
 f = chebfun2(@(x,y) cos(10*x.*y));          % Construct rhs
 u = N \ f;                                  % Solve PDE
\end{verbatim}

The final numerical solution \texttt{u} is represented in the {\sc Matlab} package {\sc Chebfun2}~\cite{Townsend_13_01} (an extension of {\sc Chebfun}~\cite{Chebfun} to
bivariate functions defined on rectangles) so that we are able to conveniently  
perform subsequent operations on the solution such as evaluation, differentiation, and integration. 
{\sc Chebfun2} represents a function by a bivariate polynomial approximation (stored in a compressed low rank form)~\cite{Townsend_13_01}. A {\em chebfun2} (in lower case letters) is any approximant constructed by {\sc Chebfun2}.  In the Helmholtz example above the solver determines that $n_x=n_y=257$ is sufficient to uniformly 
approximate the solution to~$10$ digits of accuracy. 

While our PDE solver is relatively general, it does offer the following benefits: 
\begin{itemize}
\item {\it Fast computation:} We retain the $\O((n_x n_y)^{3/2})$ complexity achieved in~\cite{Shen_95_01} for solving Poisson and Helmholtz equations, 
while allowing for general linear constraints. The same complexity extends to any linear PDE associated to an operator with a splitting rank 
of~$2$ (see Section~\ref{sec:PDErank}).
\item {\it Numerical accuracy:} The final polynomial approximant usually approximates the PDE solution to an accuracy close to machine precision relative to the absolute maximum of the solution (see Section~\ref{sec:numericalexamples}).  
\item {\it Spectral convergence with general linear constraints:} If the solution to a PDE is smooth, then there are many 
methods that achieve spectral convergence, but usually for very specific boundary conditions requiring the selection of an 
appropriate basis to be selected in advance~\cite{Julian_09_01,Shen_95_01}.  Here, our spectral method handles general linear constraints such 
as Dirichlet, Neumann, or Robin conditions in an automatic manner (see Section~\ref{sec:sylv}) and always represents the final solution in the 
tensor product Chebyshev basis.  
\item {\it Accuracy for solutions with weak singularities:} The solution to a linear PDE with smooth variable coefficients 
defined on a rectangular domain can have weak corner singularities (consider $-\nabla^2 u = 1$ with homogeneous Dirichlet boundary conditions~\cite[p.~38]{Boyd_01_01}). To globally resolve such a solution a
high degree bivariate polynomial approximation may be required.  The solver we describe is fast and numerically stable 
so high degree approximants can be reliably computed to resolve solutions with weak corner singularities.
\item {\it Automated PDE solver:} The resulting PDE solver is supplied with an anonymous function handle defining a PDO together with 
linear constraints. The discretization required to resolve the solution 
is automatically determined. The solver returns an accurate polynomial approximation of the solution represented as a chebfun2.
\end{itemize}

The PDE solver that we develop is ideal for problems where the solution is relatively smooth and the constraints on the solution 
can be written as boundary conditions. It is particularly efficient when the corresponding partial differential operator is 
of splitting rank~$1$ or~$2$ (see Section~\ref{sec:PDErankRepresentation}).  
Since the underlying discretization is a spectral method, our PDE solver should
not be used when the solution is expected to have discontinuities in low-order 
derivatives or singularities in the interior of the domain. 

The original motivation for this paper was to develop a 2D analogue of the {\sc Chebop} system~\cite{Driscoll_08_01}, which solves ordinary differential equations (ODEs) on bounded intervals in an automated 
manner using an adaptive 1D spectral collocation method. Our 2D spectral method has a different underlying methodology, but the user interface closely resembles that of its predecessor. 
In particular, the backslash command~\texttt{x = A\symbol{`\\}b} for solving 
linear systems in {\sc Matlab} that is overloaded (in the computer programming sense of the term) by {\sc Chebop} to solve linear ODEs in~\cite{Driscoll_08_01} is
now overloaded by {\sc Chebop2} for linear 2D PDEs, i.e.,~\texttt{u = N\symbol{`\\}f} (see the code snippet above). 

There are several stages of our solver that can be summarized as follows: 
\begin{enumerate}
	\item Interpret the anonymous handles for the PDO and linear constraints using automatic differentiation (see Section~\ref{sec:userInput}).
	\item Construct a separable representation (a sum of tensor products of linear ordinary differential operators) for the PDO, represent 
the ordinary differential operators with the ultraspherical spectral method (see Section~\ref{sec:ultraspherical}), and then discretize to form a generalized Sylvester matrix equation with an $n_y\times n_x$ solution matrix (see Section~\ref{sec:discretization}).
        \item Impose the linear constraints on the solution matrix and solve the resulting matrix equation using either a fast Sylvester solver for PDOs of splitting rank~$2$ or a block-banded matrix solver for operators with a splitting rank~$\geq3$ (see Section~\ref{sec:sylv}).
        \item Apply a resolution check. If the solution is unresolved in the 1st or 2nd variable, then increase $n_x$ or $n_y$ accordingly, and go back to step 2; otherwise, go to step 5. 
        \item Represent the solution as a chebfun2. 
\end{enumerate}

Figure~\ref{fig:workflow} summarizes these five stages.  Each stage is explained in more detail in subsequent sections.  
Throughout the paper we describe the spectral method for PDEs on $[-1,1]\times [-1,1]$ (to avoid the affine transformations in~\eqref{eq:bivariateChebyshevExpansion}), unless stated otherwise. The 
algorithm and software permits linear PDEs defined on bounded rectangular domains.  

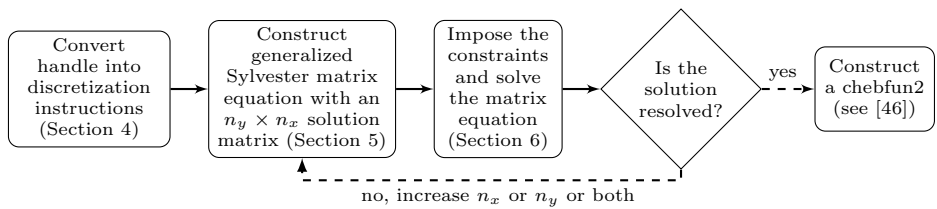
\begin{figure}
\centering
\scriptsize{
\tikzstyle{decision} = [diamond, draw, fill=white,
    text width=5em, text badly centered, node distance=2.2cm, inner sep=0pt]
\tikzstyle{block} = [rectangle, draw, fill=white,
    text width=5em, text centered, node distance=2.4cm, rounded corners, minimum height=4em]
\tikzstyle{line} = [draw, thick, color=black, -latex']
\hspace*{18pt}\begin{tikzpicture}[scale=2, node distance = 1cm, auto]
    \node [block,text width=7em] (init) {Convert handle into discretization instructions (Section~\ref{sec:userInput})};
    \node [block, right= .5cm of init,text width=8.2em] (evaluate) {Construct generalized Sylvester matrix equation with an $n_y\times n_x$ solution matrix (Section~\ref{sec:discretization})};
    \node [block, right= .5cm of evaluate,text width=5.4em] (newton) {Impose the constraints and solve the matrix equation (Section~\ref{sec:sylv})};
    \node [decision, right= .5cm of newton] (converged) {Is the solution resolved?};
    \node [block, right=.7cm of converged] (weights) {Construct a chebfun2 (see \cite{Townsend_13_01})};
    \node [, below= .3cm of newton] (no) {no, increase $n_x$ or $n_y$ or both};
    \path [line] (init) -> (evaluate);
    \path [line] (newton) -- (converged);
    \path [line] (evaluate) -- (newton);
    \path [line,dashed] (converged) -- node [, color=black,xshift=-2pt] {yes} (weights);
   \draw[|-,-|,-, dashed, thick,-latex] (converged.south)  |-+(0,-.25em)-|  (evaluate.south);
\end{tikzpicture}}
\caption{Our work-flow for solving linear PDEs defined on bounded rectangular domains. The intermediate generalized Sylvester matrix 
equations are solved in three different ways depending on their exact form (see Section~\ref{sec:matrixsolver}). }
\label{fig:workflow}\end{figure}   

In the next section we briefly describe some existing spectral methods for solving linear PDEs, and in 
Section~\ref{sec:ultraspherical} we introduce the ultraspherical spectral method.
In Section~\ref{sec:userInput} we define the splitting rank of a PDO and explain 
how it can be calculated from the anonymous handle for the operator using automatic differentiation. 
In Section~\ref{sec:discretization} we show how PDEs can be reduced to a generalized Sylvester matrix equation with linear constraints, and  
in Section~\ref{sec:sylv} we describe how to solve these constrained matrix equations. 
Finally, in Section~\ref{sec:numericalexamples} we present several numerical examples showing the generality 
of the solver before discussing possibilities for future work in Section~\ref{sec:future}. 

\begin{remark}
An experimental and rapidly developing implementation of the solver is available in the 
{\sc ApproxFun} package~\cite{ApproxFun} written in the Julia language~\cite{Julia}, which is
faster than the {\sc Matlab} implementation and supports additional bases. However, it does not currently include 
all the features described in this paper --- e.g.,  automatic differentiation
and certain splitting rank calculations --- so we 
focus on the {\sc Matlab} timings throughout, with footnotes of Julia timings for 
comparison.
\end{remark}

\section{Existing spectral methods for PDEs}
Here, we give a brief survey of spectral collocation methods~\cite{Fornberg_98_01,Trefethen_00_01}, 
spectral Galerkin methods~\cite{Shen_04_01,Julian_09_01}, spectral element methods~\cite{Patera_84_01}, and hierarchical methods~\cite{Martinsson_13_01}, as compared to the approach that we introduce. 
A more comprehensive survey can be found in~\cite{Hussaini_84_01,Kopriva_09_01}. 

\subsection{Spectral collocation methods}
Spectral collocation methods or pseudospectral methods are arguably the most convenient and widely 
applicable spectral method for PDEs. They are usually based on
tensor product grids, where the PDO is discretized by its action 
on values of an interpolating polynomial~\cite{Fornberg_98_01,Trefethen_00_01}. In 1D it is well-known that collocation methods 
lead to dense and typically ill-conditioned linear systems~\cite{Canuto_88_01}. In 2D the situation is worse as the dense linear systems are typically squared times larger in size and condition number, resulting in $\O((n_x n_y)^3)$ complexity. 
Therefore, 2D collocation methods are restricted to quite small discretization sizes~\cite{Fornberg_98_01,Trefethen_00_01}. 

Typically, collocation methods incorporate
linear constraints on the solution by {\em boundary bordering}, which replaces rows 
of a linear system by ``boundary'' rows~\cite{Boyd_01_01} that constrain the solution's values.
Sometimes it is not clear which row of the linear system should be replaced and an idea called 
rectangular spectral collocation can be used to impose boundary rows in a natural way~\cite{Driscoll_14_01}. Boundary bordering 
requires the construction of a large dense matrix. In Section~\ref{sec:userInput} we show how a separable representation of a partial 
differential operator with a splitting rank of~$2$ can be automatically computed and the associated PDE then solved by a fast 
Sylvester solver. This could be used in conjunction with a collocation method to solve some PDEs without 
constructing large ill-conditioned linear systems. Unfortunately, there is 
no convenient way to carry out boundary bordering in the matrix equation setting. 
Instead, we impose the constraints on the solution by a different, but equally general 
strategy (see Section~\ref{sec:sylv}). 

\subsection{Spectral Galerkin methods}
Spectral Galerkin methods employ global basis functions that usually depend on either the PDE, the 
linear constraints, or both. They can be derived to respect a particular underlying structure, for instance,
self-adjoint elliptic PDEs can be discretized by symmetric linear systems~\cite{Shen_95_01,Shen_96_01}. 
The resulting matrices can also be well-conditioned and block banded. For example, Shen considers several Chebyshev-based methods 
for elliptic PDEs~\cite{Shen_04_01,Shen_07_01,Shen_09_01} and
Julian and Watson employs a recombined Chebyshev basis
to achieve block banded and well-conditioned linear systems~\cite{Julian_09_01}.  

Galerkin methods usually incorporate any linear constraints by {\em basis recombination}, where the basis is constructed so that any linear 
combination is guaranteed to satisfy the constraints~\cite{Boyd_01_01}.
We find this makes Galerkin methods less applicable for a general PDE solver because designing 
the ``right'' basis is often more of an art than a science. For more exotic linear constraints there 
may not be a convenient basis readily available. 

Galerkin methods often assemble discretizations of the differential equation 
by employing a quadrature scheme that approximates the variation form of the 
equations. The way that we construct discretizations is actually equivalent, though
it does not seem so because of two fundamental differences:
(1) We do not require a quadrature rule because the integrals that appear
can be written down explicitly via recurrence relations that are satisfied by the 
orthogonal polynomials; and, (2) Our process is completely automated  
preventing us from simplifying algebraic manipulations.

\subsection{Operational tau method}\label{subsec:OperationalTau}
The operational tau method requires a tensor product orthogonal polynomial basis, where 
the PDO is discretized by its action on a matrix of 
coefficients of a bivariate polynomial~\cite{Ortiz_81_01}. In 2D the resulting linear systems are 
usually block banded from below but are otherwise dense and ill-conditioned. This approach suffers 
in a similar way to collocation methods for large discretization sizes. 

The operational tau method was popularized 
and extended by Ortiz and his colleagues~\cite{Gottlieb_77_01,Ortiz_94_01}. It is 
a useful scheme for constructing a general PDE solver, but we use the ultraspherical spectral 
method~\cite{Olver_13_01} instead, because it results in well-conditioned matrices and a 
PDE solver with a lower complexity.

\subsection{Spectral element and hierarchical methods}
Spectral element methods were introduced in~\cite{Patera_84_01} with the underlying principle of 
combining the generality of finite element methods for complex geometries with 
the accuracy of spectral methods. Typically, in 2D a domain is partitioned into rectangular regions so that on each 
subdomain the solution to a PDE 
can be represented with a low degree bivariate polynomial. Then, each subdomain is solved by a spectral method 
together with coupling conditions that impose global continuity on the solution. 
When spectral element methods are employed together with a domain decomposition method~\cite{Zanolli_87_01}, such 
as the Schwarz algorithm~\cite{Canuto_88_02} or the hierarchical Poincare--Steklov scheme~\cite{Gillman_13_01},
the resulting PDE solver has a complexity of $\mathcal{O}(N^{3/2})$ or even $\mathcal{O}(N)$, where $N$ is the total number of degrees of 
freedom used to represent the solution. Such methods always compute 
solutions that are piecewise smooth, which can lead to a suboptimal 
number of degrees of freedom required. This is particularly the case 
for highly oscillatory solutions such as those satisfying the Helmholtz equation 
with a high wavenumber. 
The spectral method we describe constructs a globally smooth approximant and hence oscillatory solutions are represented by a near-optimal number of degrees of freedom, though it currently lacks the flexibility that spectral element methods have for solving PDEs on complicated domains.  

\section{The ultraspherical spectral method}\label{sec:ultraspherical} 
A fundamental component of our PDE solver is a spectral method for linear ordinary 
differential equations (ODEs) that leads to spectrally 
accurate discretizations and almost banded\footnote{A matrix is {\em almost banded} 
if it is banded except for a small number of columns or rows.} well-conditioned matrices. 
This section reviews the ultraspherical spectral method (for further details see~\cite{Olver_13_01}). 
This will form the basis of our 2D spectral method.
 
First, consider a linear ODE with constant coefficients defined on $[-1,1]$ of the following form:
\begin{equation}
 a_N \frac{d^Nu}{dx^N} + \cdots + a_1 \frac{du}{dx} + a_0u = f, \qquad N\geq 1,  
\label{eq:univariateODE}
\end{equation}
where $a_0,\ldots,a_N$ are complex numbers, $f$ is a univariate function, and $u$ is the unknown solution. Furthermore, 
assume that 
the ODE is supplied with $K$ linear constraints, i.e., $\mathfrak{B} u = \mathbf{c}$ where $\mathfrak{B}$ is a linear operator and $\mathbf{c} \in \mathbb{C}^K$, so 
that the solution to~\eqref{eq:univariateODE} is unique. The ultraspherical spectral method aims to find the solution of~\eqref{eq:univariateODE} represented in
the Chebyshev basis and compute a vector of Chebyshev expansion coefficients of the solution.  That is, the 
spectral method seeks to find an infinite vector $\mathbf{u} = \left(u_0,u_1,\ldots\right)^T$ such that
\[
u(x) = \sum_{j=0}^{\infty} u_j T_j(x), \qquad x\in[-1,1], 
\]
where $T_j$ is the degree $j$ Chebyshev polynomial.

Classically, spectral methods represent differential operators by dense matrices~\cite{Boyd_01_01,Fornberg_98_01,Trefethen_00_01}, but the 
ultraspherical spectral method employs a ``sparse'' recurrence relation
\[
 \frac{d^\lambda T_n}{dx^\lambda} = \begin{cases} 2^{\lambda-1}n(\lambda-1)!\,C^{(\lambda)}_{n-\lambda}, & n\geq \lambda,\\ 0, & 0\leq n\leq \lambda-1, \end{cases}
\]
where $C^{(\lambda)}_{j}$ is the ultraspherical polynomial with an integer parameter $\lambda\geq 1$ of degree $j$~\cite[Sect.~18.3]{NISTHandbook}. 
This results in a sparse representation of first and higher order differential operators. The differentiation operator 
for the $\lambda$th derivative is given by
\[
\mathcal{D}_\lambda = 
2^{\lambda-1} (\lambda-1)!
\begin{pmatrix}
\overbrace{ 0\quad\cdots\quad0}^{\mbox{$\lambda$ times}}&\lambda&&&\\
& &\lambda+1&\\
& & &\lambda+2&\\
   & & & &\ddots\\
\end{pmatrix}, \qquad \lambda\geq 1.
\]
For $\lambda\geq 1$, $\mathcal{D}_\lambda$ maps a vector of Chebyshev expansion coefficients to a vector of $C^{(\lambda)}$ expansion coefficients of the $\lambda$th derivative. 
For $\lambda = 0$, $\mathcal{D}_0$ is the identity operator.

Since $\mathcal{D}_\lambda$ for $\lambda \geq 1$ returns a vector of ultraspherical expansion coefficients, 
the ultraspherical spectral method also requires {\em conversion operators}, denoted by $\mathcal{S}_\lambda$ for $\lambda\geq 0$. The operator
$\mathcal{S}_0$ converts a vector of Chebyshev coefficients to a vector of $C^{(1)}$ coefficients and, more generally, 
$\mathcal{S}_\lambda$ for $\lambda\geq 1$ converts a vector of $C^{(\lambda)}$ coefficients to a vector of $C^{(\lambda+1)}$ coefficients. 
Using the relations in~\cite[(18.9.7) and (18.9.9)]{NISTHandbook} it can be shown that (see~\cite{Olver_13_01} for a derivation)
\[
\mathcal{S}_0  = \begin{pmatrix} 
1 &  0& - \frac{1}{2} \\[3pt] 
&\frac{1}{2} &0 & -\frac{1}{2} \cr
&& \frac{1}{2} &0&\ddots \cr
&&& \frac{1}{2} &\ddots \cr
&&&&\ddots
\end{pmatrix}, \qquad 
\mathcal{S}_\lambda = \begin{pmatrix}
1&       0     & -\frac{\lambda}{\lambda+2} &               & \\[3pt]
 & \frac{\lambda}{\lambda+1}&         0     &  -\frac{\lambda}{\lambda+3} &\cr
 &            & \frac{\lambda}{\lambda+2}  &     0    & 
\ddots\cr
&&            & \frac{\lambda}{\lambda+ 3} &\ddots         \cr
&&&&\ddots
\end{pmatrix}, \quad \lambda\geq 1.
\]
Note that for $\lambda\geq 1$, the operator $\mathcal{S}_0^{-1}\cdots\mathcal{S}_{\lambda-1}^{-1}\mathcal{D}_\lambda$ is dense and upper-triangular. This is the operator that represents 
$\lambda$th order differentiation in the Chebyshev basis without converting to ultraspherical bases~\cite{Ortiz_81_01}. It is upper-triangular but otherwise dense. 

We can combine our conversion and differentiation operators to represent the ODE in~\eqref{eq:univariateODE} as follows:
\begin{equation}
\left(a_N \mathcal{D}_N +a_{N-1} \mathcal{S}_{N-1}\mathcal{D}_{N-1} + \cdots + a_0\mathcal{S}_{N-1}\cdots\mathcal{S}_0\mathcal{D}_0\right) \mathbf{u} = \mathcal{S}_{N-1}\cdots\mathcal{S}_{0}\mathbf{f}, 
\label{eq:ODErepresentation}
\end{equation}
where $\mathbf{u}$ and $\mathbf{f}$ are vectors of Chebyshev expansion coefficients of $u$ and $f$, respectively. 
The conversion operators are used in~\eqref{eq:ODErepresentation} to ensure that the resulting linear combination maps 
Chebyshev coefficients to $C^{(N)}$ coefficients, and the right-hand side $f$ is represented by a vector of $C^{(N)}$ expansion coefficients.

To make the solution to~\eqref{eq:ODErepresentation} unique we must impose the $K$ prescribed 
linear constraints in $\mathcal{B}$ on $\mathbf{u}$. That is, we must represent the action of the linear constraints on 
a vector of Chebyshev coefficients. For example, Dirichlet boundary conditions take the form
\[
\mathcal{B} = \begin{pmatrix}T_0(-1) & T_1(-1)& T_2(-1) & T_3(-1) & \cdots \\[3pt] T_0(1) & T_1(1)& T_2(1)& T_3(1) & \cdots \end{pmatrix} = \begin{pmatrix}1 & -1 & 1 & -1 &\cdots \\[3pt] 1 & 1 & 1 & 1 & \cdots\end{pmatrix},
\]
because $\mathcal{B}\mathbf{u} = (u(-1),u(1))^T$, and Neumann conditions at $x=\pm 1$ take the form 
\[
\mathcal{B} = \begin{pmatrix}T_0'(-1) & T_1'(-1)& T_2'(-1) & T_3'(-1) &\cdots \\[3pt] T_0'(1) & T_1'(1)& T_2'(1) & T_3'(1) &\cdots \end{pmatrix} = \begin{pmatrix}0 & -1 & 4 & -9 & \cdots \\[3pt] 0 & 1 & 4 & 9 & \cdots\end{pmatrix},
\]
because $\mathcal{B}\mathbf{u} = (u'(-1),u'(1))^T$. In general, any linear constraint can be represented by its action on 
a vector of Chebyshev coefficients. 

Finally, to construct a linear system that can be solved for the first $n$ Chebyshev coefficients of $u$ we take the $n\times n$ 
finite section. Let $\mathcal{P}_n$ be the {\em truncation operator} that maps
$\mathbb{C}^\infty$ to $\mathbb{C}^n$ such that $\mathcal{P}_n\mathbf{u} = \left(u_0,\dots,u_{n-1}\right)^T$.
We take the first $n$ columns of $\mathcal{B}$, $B = \mathcal{B} \mathcal{P}_n^T$, the $(n-K)\times n$ principal
submatrix of $\mathcal{L}$, $L = \mathcal{P}_{n-K}\mathcal{L}\mathcal{P}_n^T$, and form the following linear system:  
\begin{equation}
\begin{pmatrix}B\cr L\end{pmatrix}  \mathcal{P}_{n}\mathbf{u} = \begin{pmatrix}\mathbf{c}\cr \mathcal{P}_{n-K} \mathcal{S}_{N-1}\cdots\mathcal{S}_{0}\mathbf{f}\end{pmatrix}.
\label{eq:linearsystem}
\end{equation}
Since the operators $\mathcal{D}_\lambda$ and $\mathcal{S}_\lambda$ are banded, the matrix $L$ is banded, and 
the resulting linear system is almost banded, i.e., banded except for $K$ rows  
imposing the linear constraints on $\mathbf{u}$. The $K$ rows in~\eqref{eq:linearsystem} 
that impose the linear constraints could also be placed below $L$, but we  place them above so that the linear system has a structure that is as close as possible to upper-triangular. 

\subsection{Multiplication matrices}
For ODEs with variable coefficients we need to be able to represent the multiplication operation 
$\mathcal{M}[a]u = a(x)u(x)$. Since the ultraspherical spectral method converts between different 
ultraspherical bases, we need to construct multiplication matrices for each ultraspherical basis. 

Suppose we wish to represent $\mathcal{M}[a]u$, where $a(x)$ and $u(x)$ have Chebyshev expansions
\[
a(x) = \sum_{j=0}^\infty a_jT_j(x), \qquad u(x) = \sum_{j=0}^\infty u_jT_j(x),
\]
and we desire the Chebyshev expansion coefficients of $a(x)u(x)$. Define $\mathcal{M}_0[a]$ 
to be the operator that takes the vector of Chebyshev expansion coefficients of $u(x)$
and returns the vector of Chebyshev expansion coefficients of $a(x)u(x)$. It is shown in~\cite{Olver_13_01} that $\mathcal{M}_0[a]$ can be written 
as the following Toeplitz-plus-Hankel-plus-rank-$1$ operator:
\[ 
\mathcal{M}_0[a] = \frac{1}{2}
\left[
\begin{pmatrix}
2a_0& a_1 & a_2 & a_3 &\ldots\\
a_1 & 2a_0& a_1 & a_2 &\ddots\cr
a_2 &a_1 & 2a_0& a_1 &\ddots\cr
a_3 &a_2 &a_1 & 2a_0&\ddots\cr
\vdots & \ddots & \ddots &\ddots&\ddots\cr
\end{pmatrix}
+ 
\begin{pmatrix}
0& 0 & 0 & 0 &\ldots\cr
a_1 & a_2& a_3 & a_4 &\ldots\cr
a_2 &a_3 & a_4& a_5 &\iddots\cr
a_3 &a_4 &a_5 & a_6&\iddots\cr
\vdots & \iddots & \iddots &\iddots&\iddots\cr
\end{pmatrix}
\right].
\]
This multiplication operator looks dense; however, if $a(x)$ is approximated 
by a polynomial of degree $m$, then $\mathcal{M}_0[a]$ is banded with a 
bandwidth of $m$. 

In practice, we adaptively determine the degree $m$ by constructing Chebyshev interpolants
of $a(x)$ of degree $m = 8$, $m = 16$, $m=32$, and so on, until the tail of the Chebyshev coefficients 
decay to essentially machine precision.   We emphasize that the adaptive approximation of the variable coefficients is a completely independent step from the solution of the differential equation: the discretization $m$ (dictated by $a$) is independent of the discretization $n$ (dictated by $u$).  The precise adaptive algorithm we employ is 
the Chebfun constructor~\cite{Chebfun} that has many heuristic features (as it must have), 
though it is based on a decade of practical experience with function approximation.  Further 
discussion is given in~\cite[Sec.~2]{Olver_13_01}. 

We also require multiplication operators $\mathcal{M}_\lambda[a]$ that represent multiplication of two $C^{(\lambda)}$ series. That is, if 
$\mathbf{u}$ is a vector of Chebyshev expansion coefficients of $u$, then the sequence of matrices
$\mathcal{M}_\lambda[a]\mathcal{S}_{\lambda-1}\cdots\mathcal{S}_0\mathbf{u}$ returns the $C^{(\lambda)}$ expansion coefficients of $a(x)u(x)$. 
In~\cite{Olver_13_01} an explicit formula for the entries of $\mathcal{M}_\lambda[a]$ for $\lambda \geq 1$ is given and in~\cite[Chap.\ 6]{Townsend_14_02}
it is shown that $\mathcal{M}_\lambda[a]$ satisfy a three-term recurrence relation. 

Figure~\ref{fig:bandedspyplot} (left) shows the typical structure of the nonzero entries in a linear system. 
The linear system in \eqref{eq:linearsystem} can be solved in $\O(n)$ operations by the QR factorization applied to a ``filled-in'' representation~\cite{Olver_13_01}. Furthermore, an adaptive procedure based on (F.\ W.\ J.) Olver's algorithm~\cite{OlversAlgorithm} can 
be derived to find the minimum value of $n$ required to resolve the solution to machine precision with essentially no extra cost.

\begin{figure}
\begin{minipage}{.49\textwidth}
\includegraphics[width=\textwidth]{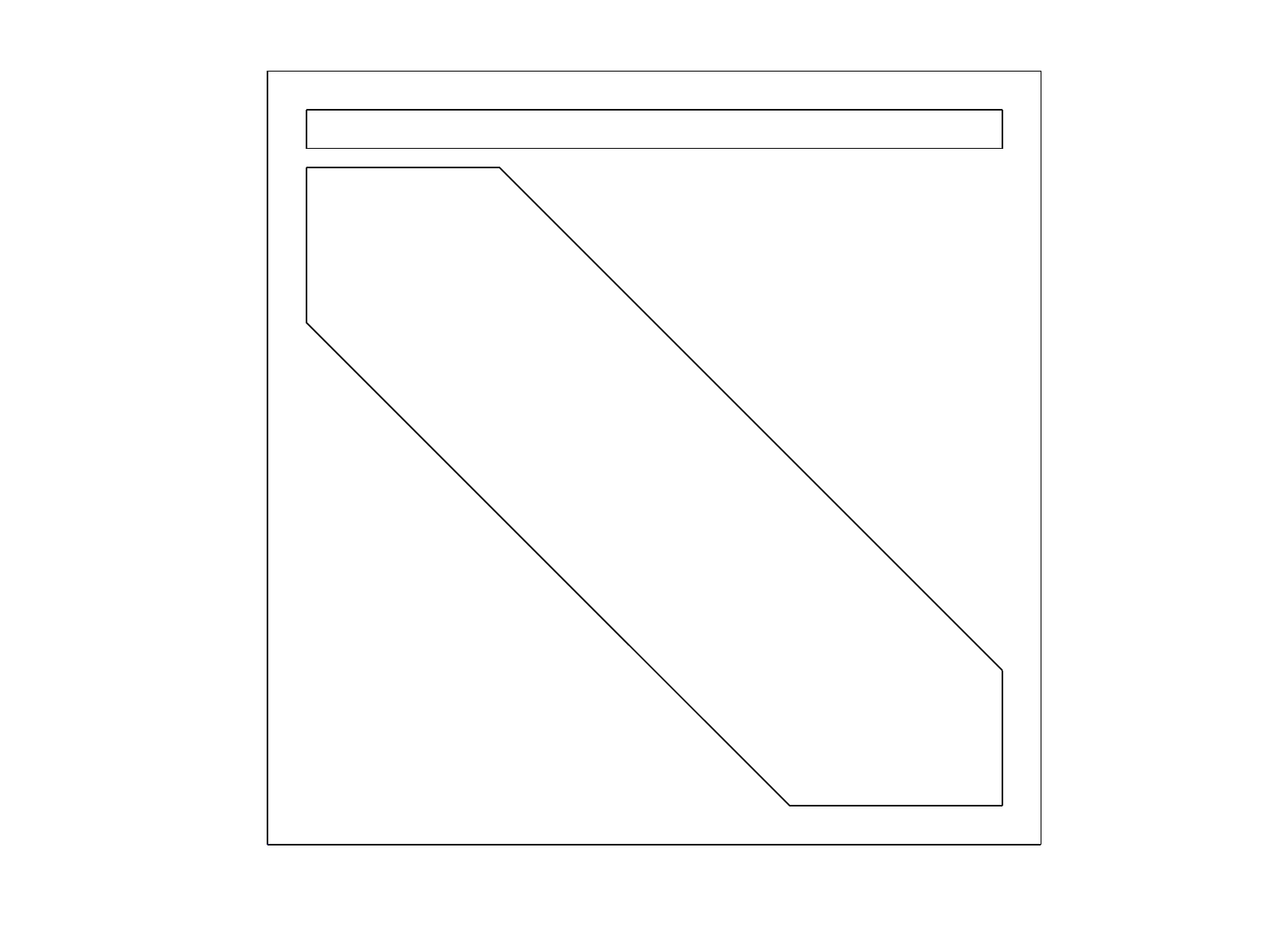}
\end{minipage}
\begin{minipage}{.49\textwidth}
\includegraphics[width=\textwidth]{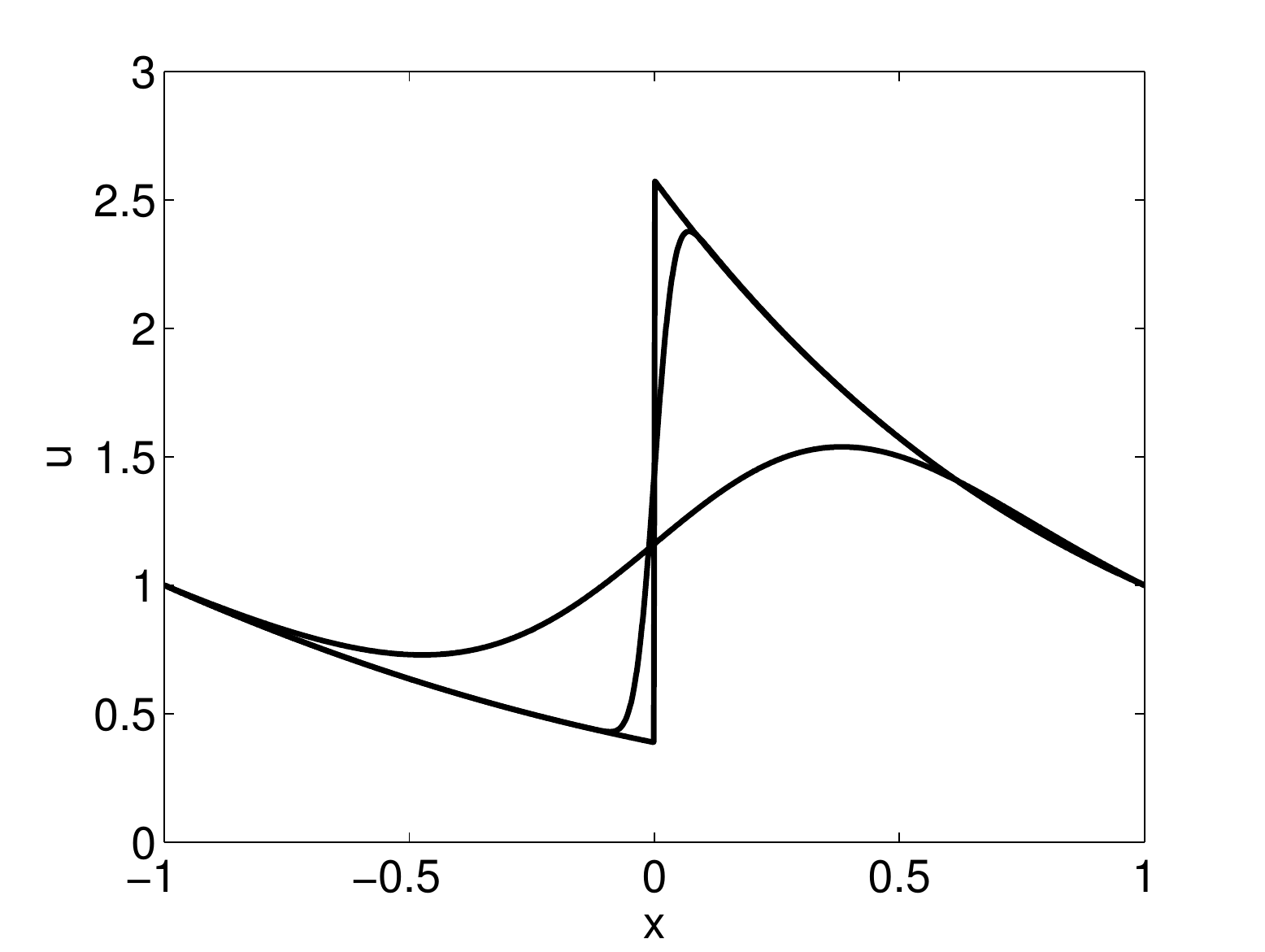}
\end{minipage}
\caption{Left: Typical structure of the matrices constructed by the ultraspherical spectral method, i.e., banded matrices except for a small number of dense rows. 
Right: The solution of $\epsilon u''(x) + xu'(x)+\sin(x)u(x) = 0$, $u(\pm 1) = 1$, 
for $\epsilon = 10^{-1}, 10^{-3}, 10^{-7}$.
The ultraspherical spectral method typically constructs well-conditioned matrices 
and hence, can resolve solutions that require large linear systems.}
\label{fig:bandedspyplot}
\end{figure}

Remarkably, this spectral method constructs not only almost banded matrices, but typically well-conditioned ones too~\cite[Lemma~4.4]{Olver_13_01}.  
Therefore, the ultraspherical spectral method is not plagued with the ill-conditioning associated to classical spectral methods. 
Figure~\ref{fig:bandedspyplot} (right) shows the solution to a singularly perturbed boundary value problem, $\epsilon u''(x) + xu'(x)+\sin(x)u(x) = 0$, $u(\pm 1) = 1$, for $\epsilon  = 10^{-1}, 10^{-3}, 10^{-7}$. 
For $\epsilon  = 10^{-7}$ a Chebyshev expansion of degree $22,\!950$ is required to approximate the solution to machine precision. 
We also observe high accuracy of the 2D spectral method we derive in this paper (see Section~\ref{sec:numericalexamples}).

\section{Automatic differentiation and separable representations}\label{sec:userInput}
We now describe the implementation and mathematics behind our 2D linear PDE solver. 
The user interface accepts input of a PDO as an anonymous handle, in a syntax 
that closely resembles how the equation is written in standard textbooks. This is achieved
in two main steps: (1) Interpret the anonymous handle for the PDO using automatic differentiation (see Section~\ref{sec:autodiff}), 
and (2) Calculate a separable representation for the PDO 
(see Sections~\ref{sec:PDErankRepresentation} and~\ref{sec:PDErank}). Once a separable representation 
has been constructed the PDE can be discretized by using the 1D ultraspherical spectral method. 

\subsection{Interpreting user-defined input using automatic differentiation}\label{sec:autodiff}
The {\sc Chebop2} interface uses automatic differentiation, more precisely, 
forward-mode operator overloading, which allows it to extract out the variable coefficients of 
a PDO given only an anonymous handle for the operator. A description of how to overload operators 
in {\sc Matlab} and implement automatic differentiation is given in~\cite{Neidinger_10_01}. 

As an example, suppose a user wants to solve a PDE with the differential equation 
$u_{xx} + u_{yy} + K^2u+yu = f$. The user could type the following into {\sc Chebop2}: 
\begin{verbatim}
 N = chebop2(@(x,y,u) diff(u,2,2) + diff(u,2,1) + k^2*u + y.*u);
\end{verbatim}
From this anonymous handle the solver derives all it needs to know about how to discretize the operator.

First, we evaluate the anonymous handle at objects \texttt{x}, \texttt{y}, and \texttt{u} 
from {\sc Matlab} classes that have their own versions of \texttt{diff}, \texttt{+}, \texttt{*}, and \texttt{.*} (the elementary operations in the anonymous handle). Then, as 
the handle is evaluated, these elementary operations are executed in a particular sequence, with each one not only computing the expected quantity but also updating an array for the variable coefficients. 
Since the individual operations are elementary, there is a simple rule on how each one should update the array of variable coefficients. 
Once complete, we have as a byproduct of the evaluation of the anonymous handle, an array containing the variable coefficients of the PDO. 
The {\sc Matlab} classes for \texttt{x}, \texttt{y}, and \texttt{u} have a growing dictionary of overloaded elementary operations so the user 
can express a PDO in a multitude of ways. A similar process is used to extract 
information from user input for the linear constraints for the PDE.  

Figure~\ref{fig:ADTree} shows how $u_{xx} + u_{yy} + K^2u+yu$ can be constructed by 
combining elementary operations. As the anonymous handle is evaluated, the tree is traversed from 
the leaves to the root node and at each node the variable coefficients of the PDO are updated. 

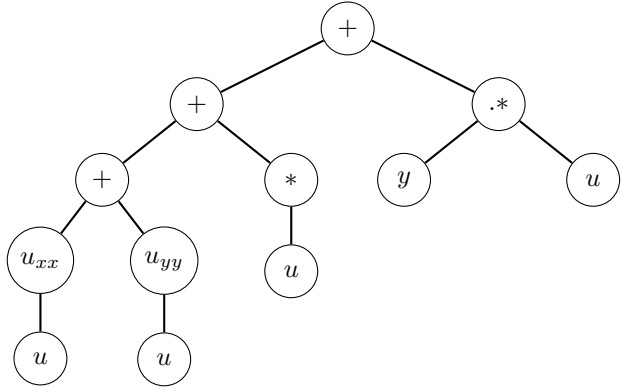
\begin{figure}
\centering
\tikzstyle{block} = [circle, draw, fill=white, minimum width=2em]
\tikzstyle{line} = [draw, thick, color=black]
\hspace*{18pt}\begin{tikzpicture}[scale=2, node distance = 1cm, auto]
    \node [block] (top) {$+$};
    \node [block, below right = .5cm and 1.5cm of top] (mult) {$.*$};
    \node [block, below left = .5cm and 1.5cm of top] (plus) {$+$};
    \node [block, below left = .5cm and .75cm of plus] (lowerplus) {$+$};
    \node [block, below right = .5cm and .75cm of plus] (vecMult) {$*$};
    \node [block, below = .5cm of vecMult] (seed3) {$u$};
    \node [block, below left = .5cm and .25cm of lowerplus] (uxx) {$u_{xx}$};
    \node [block, below right = .5cm and .25cm of lowerplus] (uyy) {$u_{yy}$};
    \node [block, below = .5cm of uxx] (seed1) {$u$};
    \node [block, below = .5cm of uyy] (seed2) {$u$};
    \node [block, below left = .5cm and .75cm of mult] (yseed) {$y$};
    \node [block, below right = .5cm and .75cm of mult] (seed4) {$u$};
    \path [line] (top) -- (mult);
    \path [line] (top) -- (plus);
    \path [line] (plus) -- (lowerplus);
    \path [line] (plus) -- (vecMult);
    \path [line] (vecMult) -- (seed3);
    \path [line] (lowerplus) -- (uxx);
    \path [line] (lowerplus) -- (uyy);
    \path [line] (uxx) -- (seed1);
    \path [line] (uyy) -- (seed2);
    \path [line] (mult) -- (yseed);
    \path [line] (mult) -- (seed4);
\end{tikzpicture}
\caption{A tree that shows how the expression $u_{xx}+u_{yy}+K^2u+yu$ can 
be constructed from $x$, $y$, $u$, and elementary operations. A tree can be traversed in an object-oriented 
language such as {\sc Matlab} by a simple
automatic differentiation technique known as forward-mode operator overloading and in the process the variable 
coefficients  for $u_{xx}+u_{yy}+K^2u+yu$ can be determined.}
\label{fig:ADTree}
\end{figure}

A one-dimensional version of this same process is described in more detail in~\cite{Birkisson_12_01}, where it is 
used to compute Fr\'{e}chet derivatives of ordinary differential equations.  Here, we are using the same technique 
except only extracting out the variable coefficients from the anonymous handle for the PDO (see~\cite{Neidinger_10_01} for more details). 

\subsection{Separable representations of partial differential operators}\label{sec:PDErankRepresentation}
A separable representation of a 2D object is a sum 
of\, ``products'' of 1D objects and in the case of linear PDOs those 1D objects are 
linear {\em ordinary differential operators} (ODOs).
We say that a linear PDO, $\mathfrak{L}$, has a splitting rank of~$\leq k$ if it can be written as a sum of $k$ tensor products of ODOs, 
\[
\mathfrak{L} = \sum_{j=1}^k \left(\mathfrak{L}_j^y\otimes \mathfrak{L}_j^x\right),
\]
and the {\em splitting rank}\footnote{Our definition of splitting rank differs from the rank of a linear operator in functional analysis. All nontrivial PDOs are of infinite rank, but usually have a finite splitting rank.} of a PDO is the minimum number of terms in such a representation.

\begin{definition} 
Let $\mathfrak{L}$ be a linear PDO in the form~\eqref{eq:PDE}. 
The splitting rank of $\mathfrak{L}$ is the smallest integer $k$ for which there exist linear ODOs 
$\mathfrak{L}^y_1,\ldots, \mathfrak{L}^y_k$ (acting on functions in $y$) and $\mathfrak{L}_1^x,\ldots, \mathfrak{L}_k^x$ (acting on functions in $x$) that satisfy
\begin{equation}
\mathfrak{L} = \sum_{j = 1}^k  \left(\mathfrak{L}^y_j\otimes \mathfrak{L}_j^x\right).
\label{eq:ranksum}
\end{equation} 
\label{def:rankoperator}
\end{definition}

A linear PDO of finite differential order with polynomial variable coefficients must itself have a 
finite splitting rank. To have an infinite splitting rank, one of its variable coefficients must be of infinite mathematical
rank (for a definition of the rank of a smooth bivariate function, see~\cite{Townsend_14_02}).  Smooth variable coefficients are approximated by polynomials, however, so  the PDEs that we consider have finite splitting rank for all practical purposes.

\subsection{Determining the splitting rank of a partial differential operator}\label{sec:PDErank}
One way to determine the splitting rank of a PDO is directly from Definition~\ref{def:rankoperator}. For example, the splitting rank of the Helmholtz operator $\partial^2/\partial x^2 + \partial^2/\partial y^2 + K^2$ is $2$ since 
\[
\partial^2/\partial x^2+\partial^2/\partial y^2+K^2 = \left(\mathfrak{I}\otimes\mathfrak{D}^2\right)+\left((\mathfrak{D}^2+K^2\mathfrak{I})\otimes\mathfrak{I}\right),
\]
where $\mathfrak{I}$ is the identity operator and $\mathfrak{D}$ is the first order differential operator. Furthermore, it can be shown that the 
splitting rank of $\partial^2/\partial x^2 + \partial^2/(\partial x\partial y) +\partial^2/\partial y^2 $ is $3$ and 
the splitting rank of $(2+\sin(x+y))\partial^2/\partial x^2 + e^{-(x^2+y^2)}\partial^2/\partial y^2$ is $4$. 
Another way, which allows it to be calculated by a computer, uses a technique motivated by umbral calculus~\cite{Bell_38_01}.
\begin{proposition} 
Let $\mathfrak{L}$ be a linear PDO in the form~\eqref{eq:PDE} with variable coefficients of finite rank. 
The splitting rank of $\mathfrak{L}$ is equal to the smallest integer $k$ required in an expression of the form
\begin{equation}
\sum_{i=0}^{N_y} \sum_{j=0}^{N_x} \ell_{ij}(s,t)y^ix^j = \sum_{j=1}^k c_j(t,y) r_j(s,x),
\label{eq:rankoperator}
\end{equation}
where $c_j$ and $r_j$ are bivariate functions. 
\label{lem:rankdiffop2}
\end{proposition}
\begin{proof}
Let $\mathcal{T}$ be the linear operator\footnote{The definition of this operator is motivated by umbral calculus~\cite{Bell_38_01}.} defined by 
\[
\mathcal{T}\left[\ell(s,t) y^i x^j \right] = \ell(x,y)\frac{\partial^{i+j}}{\partial y^i \partial x^j}, \qquad i,j\geq 0,
\]
which replaces $s$ and $t$ by $x$ and $y$ and powers of $x$ and $y$ by partial derivatives.
Now, suppose that $\mathfrak{L}$ is a linear PDO with a splitting rank of~$r$ and $k$ is the minimum number of 
terms required in~\eqref{eq:rankoperator}. We will show that $r = k$. 
 
First, note that the linear operator $\mathcal{T}$ can be used to give the following relation: 
\[
\mathfrak{L} = \sum_{i=0}^{N_y} \sum_{j=0}^{N_x} \ell_{ij}(x,y)\frac{\partial^{i+j}}{\partial y^i \partial x^j} = \mathcal{T}\left[\sum_{i=0}^{N_y} \sum_{j=0}^{N_x} \ell_{ij}(s,t)y^ix^j\right]= \mathcal{T}\left[H(s,x,t,y)\right],
\]
where $H(s,x,t,y) = \sum_{i=0}^{N_y} \sum_{j=0}^{N_x} \ell_{ij}(s,t)y^ix^j$. Now, if
the function $H(s,x,t,y)$ can be written as $\sum_{j=1}^k c_j(t,y) r_j(s,x)$, then we have 
\[
\mathfrak{L} = \mathcal{T}\left[\sum_{j=1}^k c_j(t,y) r_j(s,x)\right] = \sum_{j=1}^k \mathcal{T}\left[c_j(t,y) r_j(s,x)\right] = \sum_{j=1}^k\mathcal{T}\left[c_j(t,y)\right]\otimes \mathcal{T}\left[r_j(s,x)\right],
\]
where $\mathcal{T}[c_j(t,y)]$ and $\mathcal{T}[r_j(s,x)]$ are ODOs with variable coefficients in $y$ and $x$, respectively, and hence $r\leq k$.  
Conversely, a separable representation for $\mathfrak{L}$ can be converted (using $\mathcal{T}$) to a low rank expression for $H$, and 
hence $k\leq r$. We conclude that $r=k$ and the splitting rank of $\mathfrak{L}$ equals the minimum number of terms required in~\eqref{eq:rankoperator}.
\end{proof}

A special case of Proposition~\ref{lem:rankdiffop2} gives a connection
between constant coefficient PDOs and bivariate polynomials. This connection has been previously used to investigate
polynomial systems of equations~\cite[Chap.~10]{Sturmfels_02_01}.  In particular, if $\mathfrak{L}$ has constant 
coefficients, then the splitting rank of $\mathfrak{L}$ can be calculated as the rank of a bivariate polynomial using the singular value decomposition of a function~\cite{Townsend_14_01}.
In general, for linear PDOs with variable coefficients the splitting rank of $\mathfrak{L}$ is the 
splitting rank of a function of four variables and can be calculated using a tensor-train 
decomposition of a function~\cite{Oseledets_11_01}.   

More generally, Proposition~\ref{lem:rankdiffop2} allows us to calculate a separable representation for a linear PDO
via a low rank representation of the associated function in~\eqref{eq:rankoperator}. Each term in the separable
representation involves a tensor product of two linear ODOs, which can be discretized using the 1D ultraspherical spectral method 
(see Section~\ref{sec:ultraspherical}). In Section~\ref{sec:sylv} a PDO with a splitting rank of~$k$ will be discretized by a generalized Sylvester equation with $k$ terms.

Quite surprisingly many standard linear PDOs have a splitting rank of~$2$ 
and Table~\ref{tab:RankPDE} presents a selection. Usually, but not always, a linear
PDO with variable coefficients has a splitting rank of~$\geq3$ and any ODO is a PDO with a splitting rank of~$1$.
\begin{table} 
\centering
\begin{tabular}{c c}
\toprule
PDO & Operator  \\[3pt]
\midrule
Laplace & $u_{xx} + u_{yy}$ \\[3pt] 
Helmholtz & $u_{xx} + u_{yy} + K^2u$ \\[3pt]
Heat & $u_t - \alpha^2 u_{xx}$\\[3pt] 
Transport & $u_t - bu_x$ \\[3pt]
Wave & $u_{tt} - c^2u_{xx}$\\[3pt] 
Euler--Tricomi & $u_{xx} - xu_{yy}$\\[3pt] 
Schr\"{o}dinger & $\I \epsilon u_t + \frac{1}{2}\epsilon^2 u_{xx} - V(x) u$ \\[3pt]
Black--Scholes & $u_t + \frac{1}{2}\sigma^2x^2u_{xx}  + rxu_x - ru$ \\[3pt]
\bottomrule
\end{tabular}
\caption{A selection of PDOs with a splitting rank of~$2$ (see Definition~\ref{def:rankoperator}). 
Many constant coefficient PDOs have a splitting rank of~$2$. An exception is the biharmonic 
operator, which has a splitting rank of~$3$.}
\label{tab:RankPDE}
\end{table}

\section{Discretization of a separable representation for a partial differential operator}\label{sec:discretization} 
Any PDO with a splitting rank of~$k$ (see~\eqref{eq:ranksum}) can be discretized to a 
generalized Sylvester matrix equation with $k$ terms, $A_1XC_1^T + \cdots + A_kXC_k^T$, where the matrices 
$A_1,\ldots,A_k$ and $C_1,\ldots,C_k$ are ultraspherical spectral discretizations of ODOs 
and $X$ is a matrix containing the bivariate Chebyshev expansion coefficients of the solution.  

Specifically, suppose we seek to compute a matrix $X\in\mathbb{C}^{n_y\times n_x}$ of bivariate Chebyshev expansion coefficients of the solution $u(x,y)$ to~\eqref{eq:PDE} satisfying
\begin{equation}
\left|u(x,y) - \sum_{i=0}^{n_y}\sum_{j=0}^{n_x} X_{ij}T_i(y)T_j(x)\right|= \mathcal{O}( \epsilon \|u\|_\infty ), \qquad (x,y)\in[-1,1]^2,
\label{eq:resolved}
\end{equation} 
where $\epsilon$ is machine precision. The ultraspherical spectral method can be used to represent the ODOs $\mathfrak{L}_1^y,\ldots,\mathfrak{L}_k^y$ and $\mathfrak{L}_1^x,\ldots,\mathfrak{L}_k^x$ 
in~\eqref{eq:ranksum} as matrices $\mathcal{L}_1^y,\ldots,\mathcal{L}_k^y$ and $\mathcal{L}_1^x,\ldots,\mathcal{L}_k^x$. These 
matrices can be truncated to derive the following generalized Sylvester matrix equation: 
\begin{equation}
A_1XC_1^T + \cdots + A_kXC_k^T = F, 
\label{eq:matrixequation}
\end{equation}
where $A_j = \mathcal{P}_{n_y}\mathcal{L}_j^y\mathcal{P}_{n_y}^T$ and $C_j = \mathcal{P}_{n_x}\mathcal{L}_j^x\mathcal{P}_{n_x}^T$ for $1\leq j\leq k$, and 
$F$ is the $n_y\times n_x$ matrix of bivariate Chebyshev expansion coefficients for the right-hand side $f$ in~\eqref{eq:PDE}. 

Typically, the matrix equation~\eqref{eq:matrixequation} does not have a unique solution as the 
prescribed linear constraints $\mathfrak{B}_x$ and $\mathfrak{B}_y$ must also be incorporated.
By investigating the action of $\mathfrak{B}_x$ on the basis $\{T_0(x),\ldots,T_{n_x-1}(x)\}$, we can discretize
any linear constraint of the form $\mathcal{B}_x u(x,y) = \mathbf{g}(y)$ as 
\[
 X B_x^T = G^T,
\]
where $B_x$ is an $K_x \times n_x$ matrix and $G$ is an $K_x \times n_y$ matrix containing
the first $n_y$ Chebyshev coefficients of each component of $\mathbf{g}$. Similarly, by investigating the action of $\mathcal{B}_y$ 
on the basis $\{T_0(y),\ldots,T_{n_y-1}(y)\}$ we can discretize $\mathcal{B}_y u(x,y) = \mathbf{h}(x)$ as 
\[
B_y X = H,
\]
where $H$ is an $K_y \times n_x$ matrix containing the first $n_x$ Chebyshev coefficients of each component of $\mathbf{h}$.  

For the constraints to be consistent the matrices $B_x$ and $B_y$ must satisfy the following {\em compatibility conditions}: 
\begin{equation}
HB_x^T = (B_yX)B_x^T = B_y(XB_x^T) = B_yG^T.
\label{eq:compatibilitycondition}
\end{equation}
For example, in order that Dirichlet conditions satisfy the compatibility conditions the boundary data must match 
at the four corners of $[-1,1]^2$. Section~\ref{sec:sylv} describes how to solve matrix equations of the form~\eqref{eq:matrixequation}
with linear constraints.

In practice, the solver determines the parameters $n_x$ and $n_y$ by progressively discretizing
the PDE on finer and finer grids until the solution is resolved. First, we discretize the PDE with $n_x=n_y=9$ and solve the resulting matrix equation~\eqref{eq:matrixequation}
under linear constraints (see Section~\ref{sec:sylv}). Then, we check if the Chebyshev coefficients in $X$ decay to below machine precision relative to the maximum entry of $X$ in 
absolute value. 
Roughly speaking, if the last few columns of $X$ are above relative machine precision, then the solution has not been resolved in the $x$-variable 
and $n_x$ is increased to $17,33,65$, and so on, and likewise if the last few rows in $X$ are above relative machine precision, then $n_y$ 
is increased to $17,33,65$, and so on. The exact resolution tests we employ are the same as those employed by {\sc Chebfun2}~\cite{Townsend_13_01}, which 
are heuristic in nature, but based on a significant amount of practical experience. The discretization parameters $n_x$ and $n_y$ 
are independently increased and the resolution test is performed in both directions after each solve. Usually, this means that 
the final solution satisfies~\eqref{eq:resolved}, though it is not an absolute guarantee. 

\section{Solving matrix equations with linear constraints}\label{sec:sylv}
In this section we describe how to solve the following matrix equation with linear constraints:
\begin{equation}
\sum_{j=1}^k A_j X C_j^T = F, \quad X\in\mathbb{C}^{n_y\times n_x}, \qquad B_y X = H , \qquad X B_x^T = G^T,
\label{eq:GeneralizedmatrixEquation}
\end{equation}
where $A_j\in\mathbb{C}^{n_y\times n_y}$, $C_j\in\mathbb{C}^{n_x\times n_x}$, $F\in\mathbb{C}^{n_y\times n_x}$, $B_y\in\mathbb{C}^{K_y\times n_y}$, $B_x\in\mathbb{C}^{K_y\times n_x}$, $H\in\mathbb{C}^{K_y\times n_x}$, and $G\in\mathbb{C}^{K_x\times n_y}$.
Our approach is to use the linear constraints to remove degrees of freedom in $X$ and thus obtain 
a generalized Sylvester matrix equation with a unique solution without constraints.

By assumption the prescribed linear constraints are linearly independent so 
the column ranks of $B_x$ and $B_y$ are $K_x$ and $K_y$, respectively.  
Without loss of generality, we further assume that the principal $K_x\times K_x$ and 
$K_y\times K_y$ submatrices of $B_x$ and $B_y$ are the identity matrices\footnote{Otherwise, permute the columns of $B_x$ and $B_y$,  and the corresponding rows/columns of $X$,
so the principal $K_x\times K_x$ and $K_y\times K_y$ matrices $\hat{B}_x$ and $\hat{B}_y$ are invertible, then redefine as $B_x \mapsto \hat{B}_{x}^{-1}B_x$, $G \mapsto\hat{B}_x^{-1}G$, $B_y \mapsto \hat{B}_{y}^{-1}B_y$, and $H \mapsto\hat{B}_{y}^{-1}H$.} $I_{K_x}$ and $I_{K_y}$. 
Then, we can modify the matrix equation in~\eqref{eq:GeneralizedmatrixEquation} to
\[
\sum_{j=1}^k A_j X C_j^T  - \sum_{j=1}^k (A_j)_{1:n_y,1:K_y} B_y X C_j^T = F - \sum_{j=1}^k (A_j)_{1:n_y,1:K_y} H C_j^T,
\]
where we have used the constraint $B_y X = H$. Moreover, by rearranging we have
\[
\sum_{j=1}^k A_j X C_j^T  - \sum_{j=1}^k (A_j)_{1:n_y,1:K_y} B_y X C_j^T = \sum_{j=1}^k  \left(A_j - (A_j)_{1:n_y,1:K_y} B_y\right) X C_j^T,
\]
and since the $K_y\times K_y$ principal matrix of $B_y$ is the identity matrix, each matrix $A_j - (A_j)_{1:n_y,1:K_y} B_y$ for 
$1\leq j\leq k$ is zero in the first $K_y$ columns. Similarly, 
the condition $XB_x^T = G^T$ can be used to further modify the matrix equation as follows:
\begin{equation}
\begin{aligned}
\sum_{j=1}^k  &\left(A_j - (A_j)_{1:n_y,1:K_y} B_y\right) X  \left(C_j - B_x(C_j)_{1:n_x,1:K_x}\right)^T\\
& = F - \sum_{j=1}^k (A_j)_{1:n_y,1:K_y} H C_j^T - \sum_{j=1}^k \left(A_j - (A_j)_{1:n_y,1:K_y} B_y\right) G^T (C_j)_{1:n_x,1:K_x}^T, 
\end{aligned}
\label{eq:matrixReducedEquation}
\end{equation}
so that the matrices $(C_j - B_x(C_j)_{1:n_x,1:K_x})^T$ for $1\leq j\leq k$ are zero in the first $K_x$ rows.

Now, the first $K_y$ columns of $A_j - (A_j)_{1:n_y,1:K_y} B_y$ and the first $K_x$ rows of $(C_j - B_x(C_j)_{1:n_x,1:K_x})^T$ are zero in~\eqref{eq:matrixReducedEquation}
and hence, the matrix equation is independent of the first $K_y$ rows and $K_x$ columns of $X$. Therefore, the matrix equation in~\eqref{eq:matrixReducedEquation}
can be reduced by removing those columns and rows and then solved, obtaining a matrix $X_{22}\in\mathbb{C}^{(n_y-K_y)\times (n_x-K_x)}$, where 
\[
X = \begin{pmatrix} X_{11} & X_{12} \cr X_{21} & X_{22}\end{pmatrix}, \quad X_{11}\in\mathbb{C}^{K_y\times K_x}, \quad X_{12}\in\mathbb{C}^{K_y\times (n_x-K_x)},\quad X_{21}\in\mathbb{C}^{(n_y-K_y)\times K_x}. 
\]
The solution of the resulting unconstrained generalized Sylvester equation that $X_{22}$ satisfies is given in Section~\ref{sec:matrixsolver}. 

Once we have computed $X_{22}$ we can recover $X$ by using the linear constraints. For instance, since $B_yX = H$ 
and the $K_y\times K_y$ principal submatrix of $B_y$ is the identity matrix, we have 
\[
X_{12} = H_2 - B_y^{(2)}X_{22},
\]
where $H = \left[H_1,H_2\right]$ with $H_1\in\mathbb{C}^{K_y\times K_x}$ and $H_2\in\mathbb{C}^{K_y\times (n_x-K_x)}$, and $B_y = [I_{K_y},B_y^{(2)}]$ with $B_y^{(2)}\in\mathbb{C}^{K_y\times (n_y-K_y)}$. 
Furthermore, since $XB_x^T = G^T$ and the $K_x\times K_x$ principal submatrix of $B_x$ is the identity matrix, we have 
\[
X_{21} = G_2^T - X_{22}(B_x^{(2)})^T,
\]
where $G = \left[G_1,G_2\right]$ with $G_1\in\mathbb{C}^{K_x\times K_y}$ and $G_2\in\mathbb{C}^{K_x\times (n_y-K_y)}$, and $B_x = [I_{K_x},B_x^{(2)}]$ with $B_x^{(2)}\in\mathbb{C}^{K_x\times (n_x-K_x)}$. 
Lastly, we can recover $X_{11}$ using either of the two formulas
\[
X_{11} = H_1 - B_y^{(2)}X_{21}, \qquad X_{11} = G_1^T - X_{12}(B_x^{(2)})^T,
\]
since the compatibility condition~\eqref{eq:compatibilitycondition} ensures that both formulas are equivalent.

\subsection{Solving a generalized Sylvester matrix equation}\label{sec:matrixsolver}
We are left with a standard generalized Sylvester matrix equation of the form
\begin{equation}
\sum_{j=1}^k \tilde{A}_j X_{22} \tilde{C}_j^T  = \tilde{F},
\label{eq:standardMatrixEquation}
\end{equation}
and the exact algorithm we use to solve for $X_{22}$ depends on $k$. 

If $k = 1$ then the matrix equation takes the form $\tilde{A}_1 X_{22} \tilde{C}_1^T = \tilde{F}$, and 
since we are using the ultraspherical spectral method (see Section~\ref{sec:ultraspherical}) the 
matrices $\tilde{A}_1$ and $\tilde{C}_1$ are almost banded. Therefore, we can solve $\tilde{A}_1 Y = \tilde{F}$ for $Y\in\mathbb{C}^{(n_y-K_y)\times (n_x-K_x)}$ 
in $\mathcal{O}(n_x n_y)$ operations and then solve $\tilde{C}_1 X_{22}^T = Y^T$ for $X_{22}$ in $\mathcal{O}(n_x n_y)$ operations using the adaptive QR method~\cite{Olver_13_01}.  

If $k=2$ then the matrix equation takes the form
\begin{equation}
 \tilde{A}_1 X_{22} \tilde{C}_1^T  +  \tilde{A}_2 X_{22} \tilde{C}_2^T= \tilde{F}.
\label{eq:matrixEquationRankTwo}
\end{equation}
To solve~\eqref{eq:matrixEquationRankTwo} we use the generalized Bartels--Stewart algorithm~\cite{Bartels_72_01,Gardiner_92_01}, 
which requires $\mathcal{O}(n_x^3 + n_y^3)$ operations. Alternatively, the generalized Hessenberg--Schur 
algorithm can be used~\cite{Gardiner_92_01} or the recursive blocked algorithms in RECSY (see~\cite{Jonsson_03_01}). It turns out that many standard PDOs with constant coefficients have 
a splitting rank of~$2$ (see Table~\ref{tab:RankPDE}).

For $k\geq 3$, we are not aware of an efficient algorithm for solving~\eqref{eq:standardMatrixEquation}. Instead,
we expand the matrix equation into an $(n_x-K_x)(n_y-K_y)\times (n_x-K_x)(n_y-K_y)$ linear system 
\begin{equation}
\left( \sum_{j=1}^k (\tilde{C}_j \otimes \tilde{A}_j) \right)\hbox{vec}(X_{22}) = \hbox{vec}(\tilde{F}),
\label{eq:linearsystemGeneral}
\end{equation} 
where `$\otimes$' denotes the Kronecker product operator for matrices and $\hbox{vec}(C)$ denotes the vectorization 
of the matrix $C$ formed by stacking the columns of $C$ into a single column vector.

Na\"\i vely solving the resulting linear system~\eqref{eq:linearsystemGeneral} requires $\mathcal{O}( (n_xn_y)^3 )$ operations.  
However, because we are using the ultraspherical spectral method the matrices $\tilde{A}_j$ and $\tilde{C}_j$ are almost banded and hence, the matrix $\sum_{j=1}^k (\tilde{C}_j \otimes \tilde{A}_j)$
is also almost banded with a bandwidth of $\mathcal{O}(n_x)$ except for $\mathcal{O}(n_x)$ dense rows. Thus, the linear system can be solved in 
$\mathcal{O}(n_x^2 (n_x n_y)) = \mathcal{O}(n_x^3 n_y)$ operations using the adaptive QR method~\cite{Olver_13_01}. Alternatively, 
the roles of $x$ and $y$ can be swapped and the linear system solved in $\mathcal{O}(n_x n_y^3)$ operations.

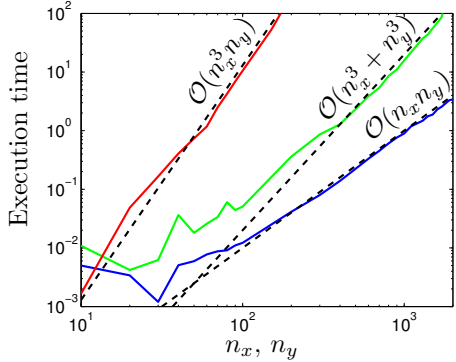
\begin{figure} 
\centering
\begin{overpic}[width=.49\textwidth]{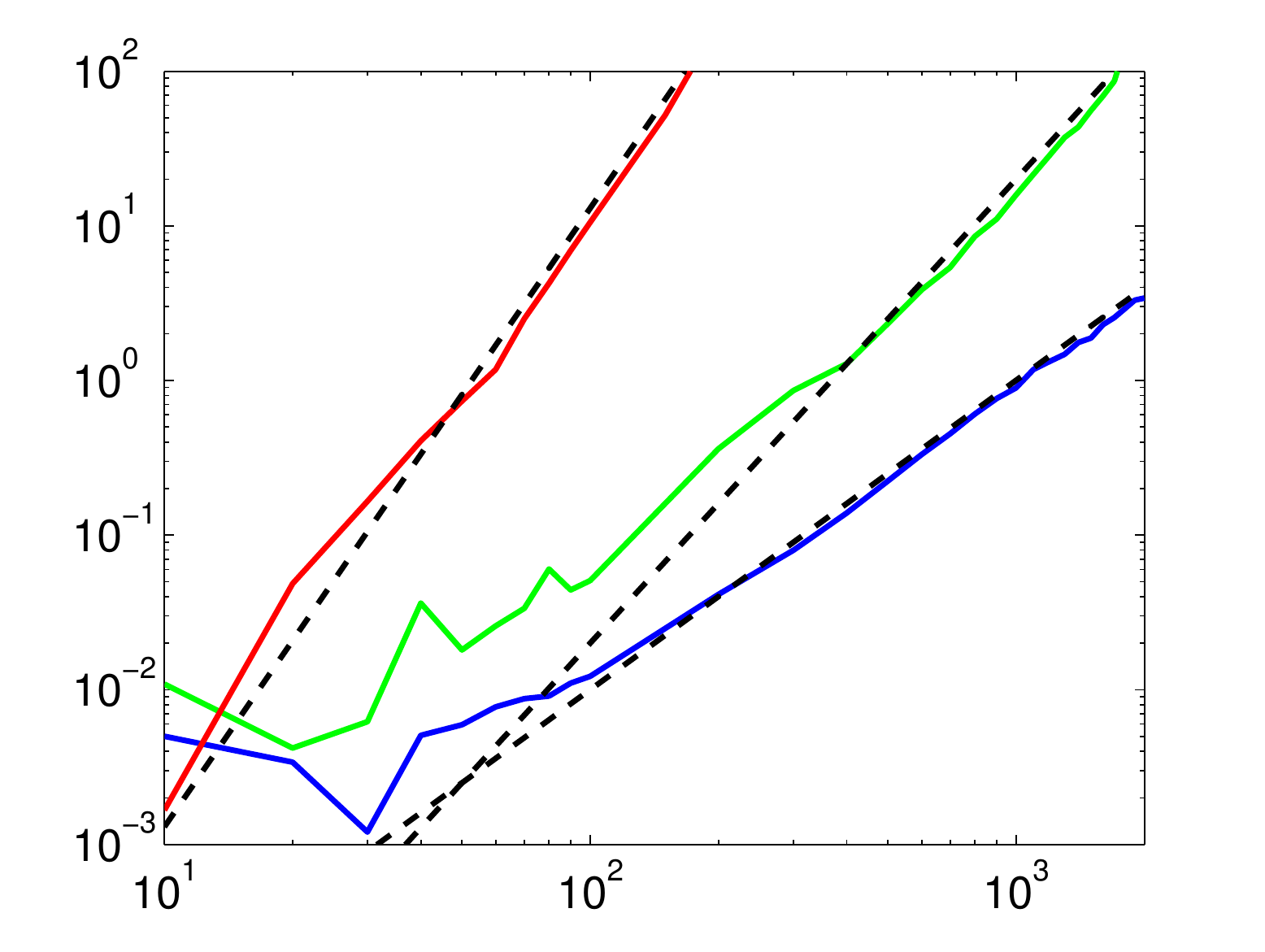} 
\put(61,47){\rotatebox{48}{$\mathcal{O}(n_x^3 + n_y^3)$}}
\put(71,43){\rotatebox{36}{$\mathcal{O}(n_xn_y)$}}
\put(34,50){\rotatebox{57}{$\mathcal{O}(n_x^3n_y)$}}
\put(43,-1){$n_x$, $n_y$}
\put(-2,25){\rotatebox{90}{Execution time}}
\end{overpic}
\caption{Computation cost and complexity for solving the matrix equations $\tilde{A}_1 X_{22} \tilde{C}_1^T = \tilde{F}$,~\eqref{eq:matrixEquationRankTwo}, 
and~\eqref{eq:linearsystemGeneral}, where $\tilde{A}_j\in\mathbb{C}^{(n_y-K_y)\times (n_y-K_y)}$ and $\tilde{C}_j\in\mathbb{C}^{(n_x-K_x)\times (n_x-K_x)}$
are almost banded with a bandwidth of $10$. 
Solving~\eqref{eq:GeneralizedmatrixEquation} is the dominating computational cost when solving PDO of splitting rank~$1$, $2$, and $k\geq3$.
}
\label{fig:complexity}
\end{figure}	

Figure~\ref{fig:complexity} shows the computational time for solving $\tilde{A}_1 X_{22} \tilde{C}_1^T = \tilde{F}$, 
~\eqref{eq:matrixEquationRankTwo}, and~\eqref{eq:linearsystemGeneral}, where the matrices are almost banded with a
bandwidth of $10$. The typical dominating computational cost of the solver for PDOs with splitting rank~$1$, ~$2$, 
and $k\geq 3$ is the matrix equation solve. In particular, Figure~\ref{fig:complexity} shows the substantial 
efficiency gain that can be achieved when the splitting rank structure of a PDO is exploited.

\begin{remark}
The Haidvogel--Zang algorithm~\cite{Haidvogel_79_01} (also see~\cite[Chap.~15]{Boyd_01_01}) solves the Helmholtz equation by diagonalizing the operator in only one direction 
and applying an $\O(n)$ solver in the remaining direction.  This fits naturally into the proposed framework, using the generalized 
Schur decomposition in one dimension and exploiting the almost-banded structure of the ultraspherical discretization in the other.  A variant of 
this idea is used in the Julia implementation~\cite{ApproxFun} by applying the adaptive QR algorithm to determine the appropriate 
discretization size~\cite{Olver_14_01}.
\end{remark}

\subsection{Solving subproblems}
If the even and odd modes of the solution decouple, then the computational cost can be reduced 
by solving for them separately.  For example, Laplace's equation with Dirichlet conditions can be split into four subproblems since 
the PDO contains only even order derivatives in $x$ and $y$ and the boundary conditions can be equivalently written 
as
\[
\mathcal{B}_x = \begin{pmatrix}1&0&1&0&1&\cdots\\[3pt]0&1&0&1&0&\cdots\end{pmatrix},\qquad \mathcal{B}_y = \begin{pmatrix}1&0&1&0&1&\cdots\\[3pt]0&1&0&1&0&\cdots\end{pmatrix}.
\]
This means that the even and odd modes decouple and in this 
case, since the Laplace operator has a splitting rank of~$2$, the computational cost is 
reduced by a factor of $8$ by solving four subproblems. 

In fact, any PDO with constant coefficients that contains only even (or odd) order derivatives 
in one variable accompanied with pure Dirichlet or pure Neumann boundary conditions 
decouples into two subproblems. Moreover, if it contains only even (or odd) 
order derivatives in both variables then it decouples into 
four subproblems.
Our implementation automatically detects these cases and splits the problem into two or four subproblems as appropriate.

In principle, higher order symmetries (see~\cite[Chap.~9]{Boyd_01_01} and~\cite{Li_14_01}) could be detected and exploited by our solver.
However, we have decided not to do this because such symmetries appear less often in practice. 

\section{Numerical examples}\label{sec:numericalexamples}
We now demonstrate our 2D spectral method on five examples. A {\sc Matlab} implementation is 
available as part of {\sc Chebfun}~\cite{Chebfun} via the \texttt{chebop2} command. An experimental implementation 
is also available in the {\sc ApproxFun} package~\cite{ApproxFun} written in the Julia language~\cite{Julia}, 
and timings are given when available for comparison.

\subsection*{Example 1: The Helmholtz equation}
First, we consider the Helmholtz equation $u_{xx}+u_{yy}+K^2u=0$ on $[-1,1]^2$ with Dirichlet boundary conditions, where 
$K$ is some wavenumber. This simple example is used to verify that our global spectral method resolves
oscillatory solutions with an average of $\pi$ degrees of freedom per wavelength. In particular, we set $K = \sqrt{2}\omega$ and solve
\begin{equation}\label{eq:helmholtz}
u_{xx} + u_{yy} + (\sqrt{2}\omega)^2 u = 0, \qquad u(\pm 1,y) = f(\pm 1,y), \quad u(x,\pm 1) = f(x,\pm 1),
\end{equation}
where $\omega\in\mathbb{R}$ and $f(x,y) = \cos(\omega x)\cos(\omega y)$.  The exact solution is $u=f$. In Figure~\ref{fig:Helmholtz} we plot the solution for $\omega = 50$ and plot 
the Cauchy error for $\omega = 10\pi, 50\pi, 100\pi$. The Cauchy error shows that the solution 
is rapidly resolved once $\pi$ degrees of freedom per wavelength are used (in agreement with the Shannon--Nyquist sampling rate~\cite{Shannon_98_01}). 

For $\omega = 100\pi$ in~\eqref{eq:helmholtz} we have 
\[
\left(\int_{-1}^1 \int_{-1}^1 \left(\tilde{u}(x,y) - u(x,y)\right)^2 dxdy \right)^{\frac{1}{2}} = 5.44\times 10^{-10},
\]
where $u$ is the exact solution and $\tilde{u}$ is the computed solution. This error is relatively 
small considering that the solution has more than $20,\!000$ local extrema in $[-1,1]^2$. The
solution $\tilde{u}$ was computed in\footnote{Experiments were performed on a 2012 1.8GHz Intel Core i7 MacBook Air 
with {\sc Matlab} 2012a.} $6.06$ seconds.  (The Julia implementation takes $3.90$ seconds.) The implementation automatically set up subproblems, which reduced the 
computational time by a factor of about $8$. 

\begin{figure} 
\centering 
\begin{minipage}{.49\textwidth} 
\centering 
\includegraphics[width=\textwidth]{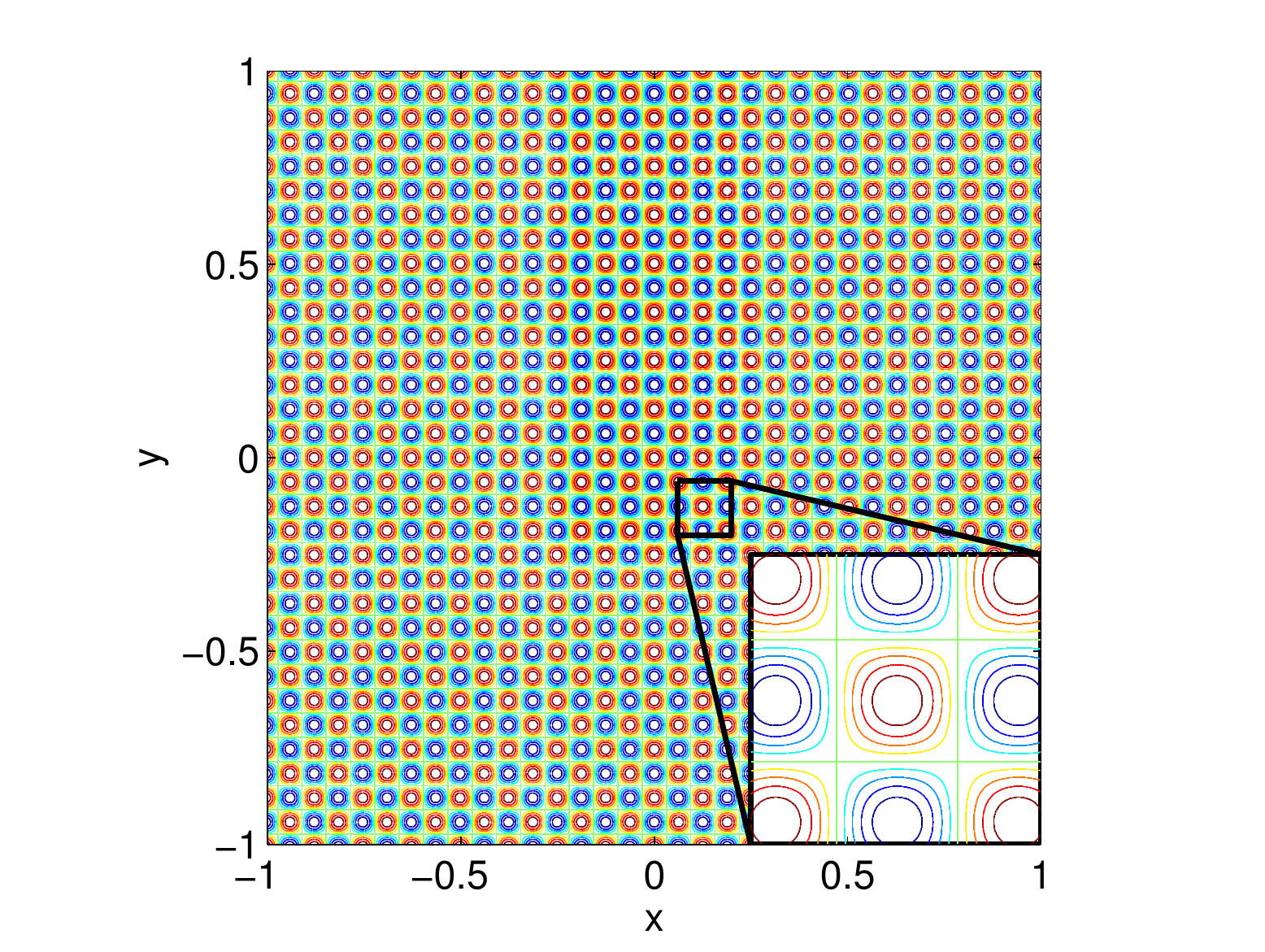}
\end{minipage}
\begin{minipage}{.49\textwidth} 
\centering 
\begin{overpic}[width=\textwidth]{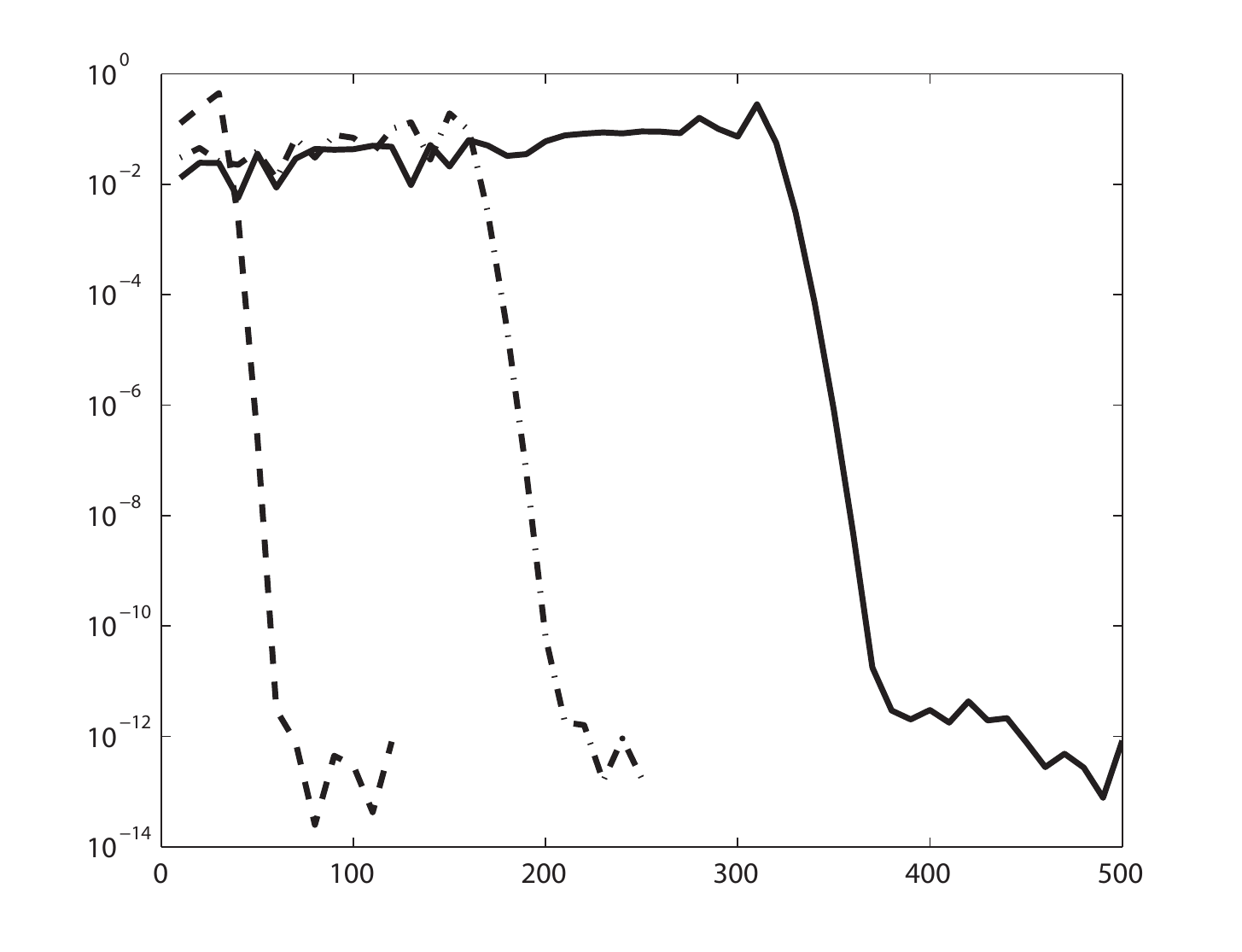}
\put(0,20){\footnotesize{\rotatebox{90}{$\left\| u_{\lceil 1.01n \rceil} - u_{n} \right\|_2$}}}
\put(50,0){\footnotesize{$n$}}
\put(22,40){\rotatebox{275}{\scriptsize{$\omega\!=\!10\pi$}}}
\put(44,40){\rotatebox{275}{\scriptsize{$\omega\!=\!50\pi$}}}
\put(69,40){\rotatebox{278}{\scriptsize{$\omega\!=\!100\pi$}}}
\end{overpic}
\end{minipage}
\caption{Left: Solution of~\eqref{eq:helmholtz} for $\omega = 50$. Right: Cauchy error for 
the solution's coefficients for $\omega = 10\pi$ (dashed), $\omega=50\pi$ (dot-dashed), and $\omega=100\pi$ (solid), which shows the 2-norm difference between the coefficients of the 
approximate solution when computed from an $n\times n$ and an $\lceil 1.01 n\rceil \times \lceil 1.01n \rceil$ discretization.}
\label{fig:Helmholtz} \end{figure}
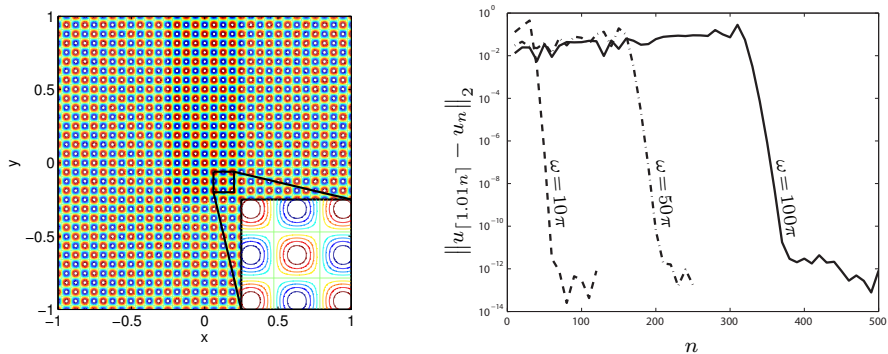

The convergence behavior for this example is not typical for Helmholtz equations because 
the solution does not contain a weak corner singularity. Figure~\ref{fig:VariableHelmholtz} (left)
shows the more typical Cauchy error plot for Helmholtz equations: For low 
discretization sizes there is no decay of the Cauchy error (more degrees of freedom are required 
to reach Nyquist's sampling rate), followed by a short-lived but rapid geometric or super-geometric 
decay (resolving the smooth part of the solution), and then a slower algebraic decay of the error 
(resolving the weak corner singularity of the solution). Our 2D spectral method allows for quite 
large discretization sizes, so despite only algebraic decay the solution can still be resolved 
to a high accuracy.

\subsection*{Example 2: A variable coefficient Helmholtz equation}
Next, to make the Helmholtz equation more challenging we add a variable wave number and a 
forcing term.  Consider $\nabla^2 u + k(x,y)u = f(x,y)$ on $[-1,1]^2$, 
where $k(x,y) = (x^2+(y+1)^2)\sin(x(y+1))^2$ and $f(x,y) = (x^2+(y+1)^2)\cos(x(y+1))\sin(\cos(x(y+1)))$
with Dirichlet data so that the solution is $u(x,y) = \cos(\cos(x(y+1)))$.  
The PDE has an operator that has an unbounded splitting rank; however, numerically 
the operator can be well-approximated by an operator with a splitting rank of $9$ (the exact separable approximation 
is calculated using Proposition~\ref{lem:rankdiffop2} and the tensor-train decomposition). This structure can then be used to discretize the PDE as a 
generalized Sylvester matrix equation involving $9$ terms of the form~\eqref{eq:matrixEquationRankTwo}.  
All this happens automatically and the PDE can be solved in \texttt{chebop2} with the following syntax: 
\begin{verbatim} 
N = chebop2(@(x,y,u) lap(u)+(x.^2+(y+1).^2).*sin(x.*(y+1)).^2.*u); 
N.lbc = @(y) cos(cos(-y-1)); N.rbc = @(y) cos(cos(y+1));
N.dbc = @(x) cos(cos(-x-1)); N.ubc = @(x) cos(cos(x+1));
f = chebfun2(@(x,y) (x.^2+(y+1).^2).*cos(x.*(y+1)).*sin(cos(x.*(y+1))));
u = N \ f; 
\end{verbatim} 

Figure~\ref{fig:VariableHelmholtz} (right) shows a surface plot of the solution, which is calculated 
to an accuracy of $14$-digits. PDOs with splitting rank $\geq 3$ are solved less efficiently because 
the linear algebra required to solve the matrix equation is more expensive requiring large block 
almost banded matrices constructed by multiplying out Kronecker products. However, the underlying 
automated process for constructing discretization and solving the resulting 
matrix equation is applicable to any variable 
coefficient PDO. 
%
%
\begin{figure} 
\centering 
\begin{minipage}{.49\textwidth} 
\centering 
\begin{overpic}[width=\textwidth]{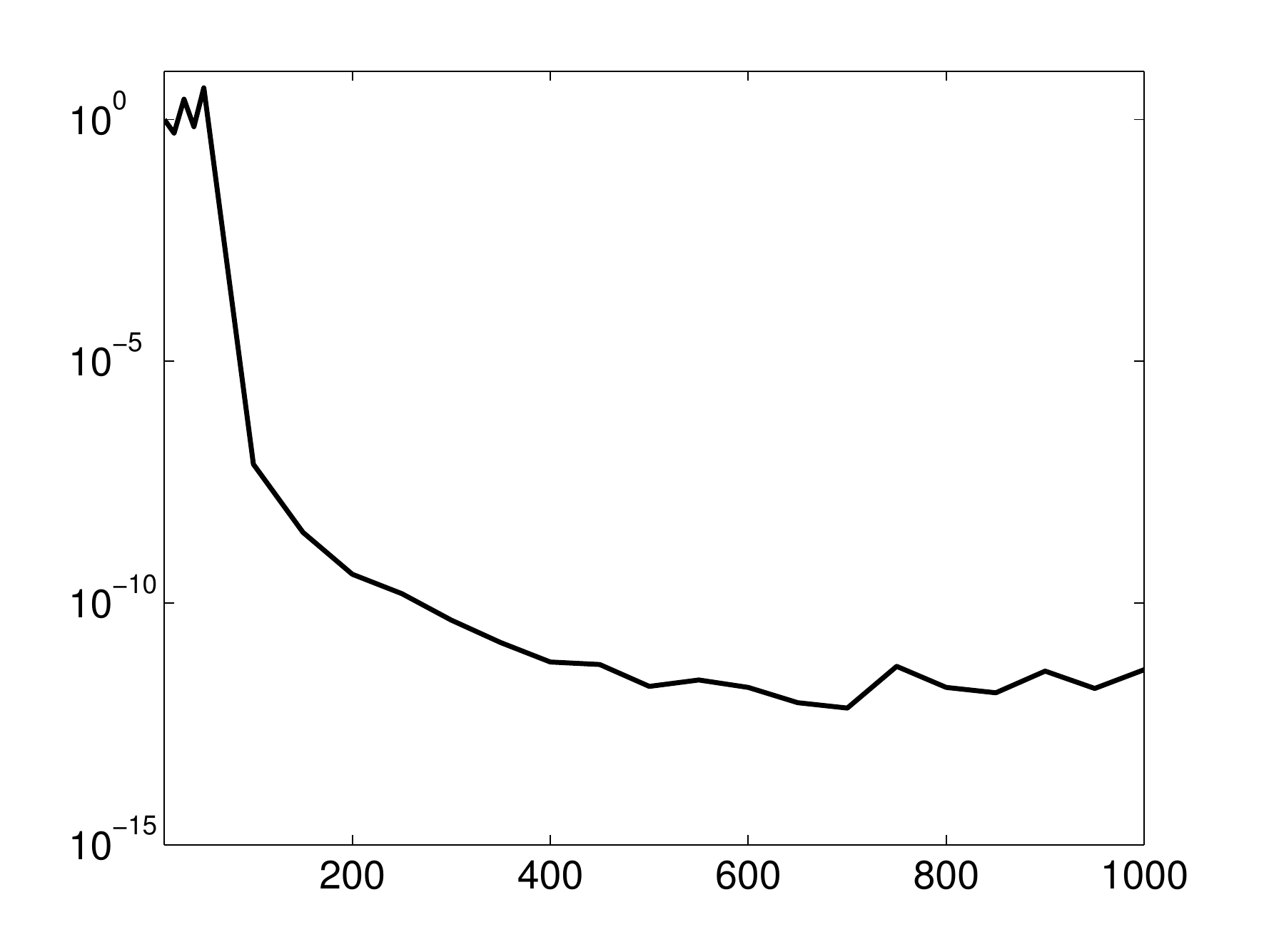}
\put(0,20){\footnotesize{\rotatebox{90}{$\left\| u_{\lceil 1.01n \rceil} - u_{n} \right\|_2$}}}
\put(50,0){\footnotesize{$n$}}
\end{overpic}
\end{minipage}
\begin{minipage}{.49\textwidth} 
\centering 
\begin{overpic}[width=\textwidth]{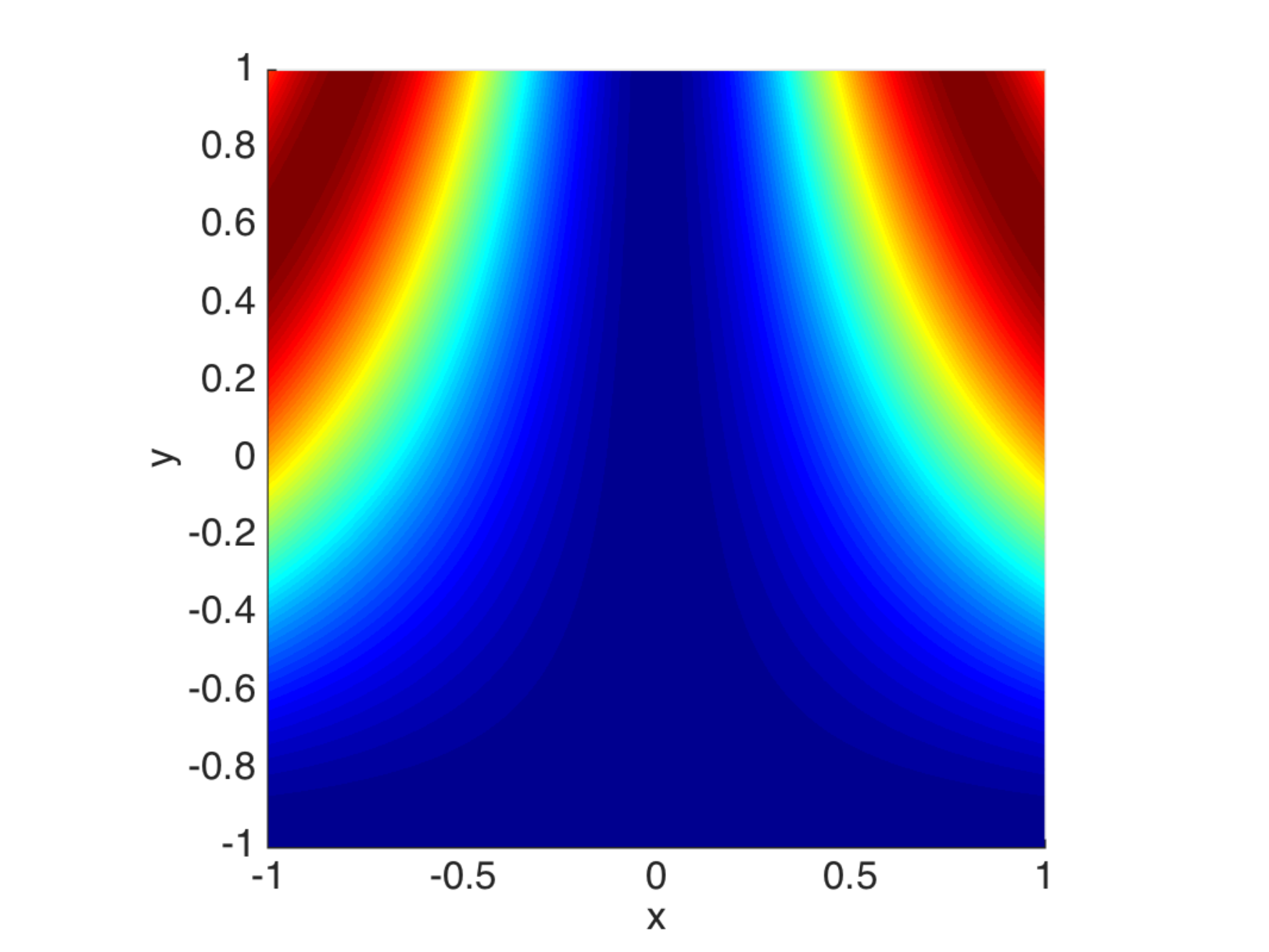}
\end{overpic}
\end{minipage}
\caption{
Left: Typical Cauchy error plot for a Helmholtz solution's coefficients, showing 
the 2-norm difference between the coefficients of the 
approximate solution when discretized by an $n\times n$ 
and $\lceil 1.01 n\rceil \times \lceil 1.01n \rceil$ matrix equation. 
Usually, the solution has a weak corner singularity that appears as 
algebraic convergence of the Cauchy error.
Right: The solution to $\nabla^2 u + k(x,y)u = f(x,y)$ on $[-1,1]^2$, 
where $k(x,y) = (x^2+(y+1)^2)\sin(x(y+1))^2$ and 
$f(x,y) = (x^2+(y+1)^2)\cos(x(y+1))\sin(\cos(x(y+1)))$
with Dirichlet data so that the solution is $u(x,y) = \cos(\cos(x(y+1)))$.}
\label{fig:VariableHelmholtz} \end{figure}

\subsection*{Example~3: The wave equation and the Klein--Gordon equation}
Next we consider the wave equation $u_{tt} = u_{xx}$ modeling a string of length~$2$ initially in a moment of time symmetry  $u_t(x,0)=0$ with displacement  $u(x,0) = e^{-50(x-2/10)^2}$, held fixed on the left $u(-1,t) =0$,  and held by a vertical elastic band on the right $u(1,t) + 5u_x(1,t) =0$. The string is left
to vibrate freely for $10$ units of time. We will compare this solution to that of the Klein--Gordon equation $u_{tt}=u_{xx}-5u$ with the same boundary conditions. The latter equation can be solved by the following {\sc Chebop2} code: 
\begin{verbatim} 
 N = chebop2(@(u) diff(u,2,1) - diff(u,2,2) + 5*u, [-1 1 0 10]); 
 N.lbc = 0; N.rbc = @(t,u) u/5 + diff(u); 
 N.dbc = @(x,u) [u-exp(-50*(x-.2).^2) ; diff(u)];
 u = N \ 0;
\end{verbatim} 

In Figure~\ref{fig:WaveGordon} we plot the solutions side-by-side. It can be seen that 
the solution to the wave equation (left) has the initial pulse traveling at a constant speed 
reflecting with equal and opposite amplitude off the left and with 
equal sign (but not quite equal amplitude) from the right. This is typical  
reflection behavior of traveling waves with these boundary conditions. In contrast, in the solution to the Klein--Gordon 
equation (right) high frequencies of the pulse travel faster than low frequencies 
and interference quickly destroys any regular pattern. We require about $2.02$ seconds
to resolve the Klein--Gordon solution to $8$-digits of accuracy with a $(92,\!257)$ degree bivariate polynomial.   (The Julia implementation takes $0.89$ seconds.)
\begin{figure} 
\centering 
\begin{minipage}{.49\textwidth} 
\centering 
\begin{overpic}[width=\textwidth]{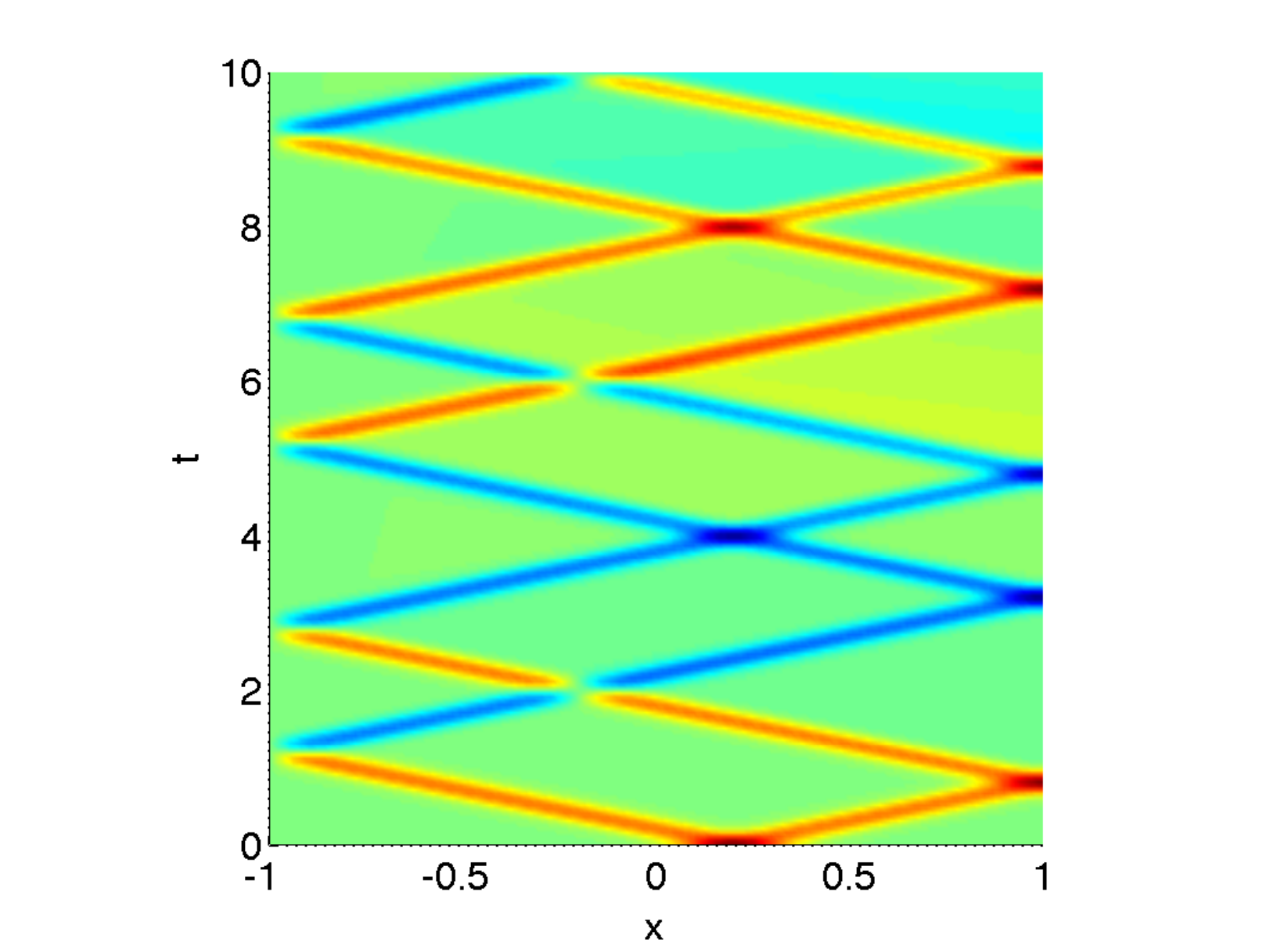}
\end{overpic}
\end{minipage}
\begin{minipage}{.49\textwidth} 
\centering 
\begin{overpic}[width=\textwidth]{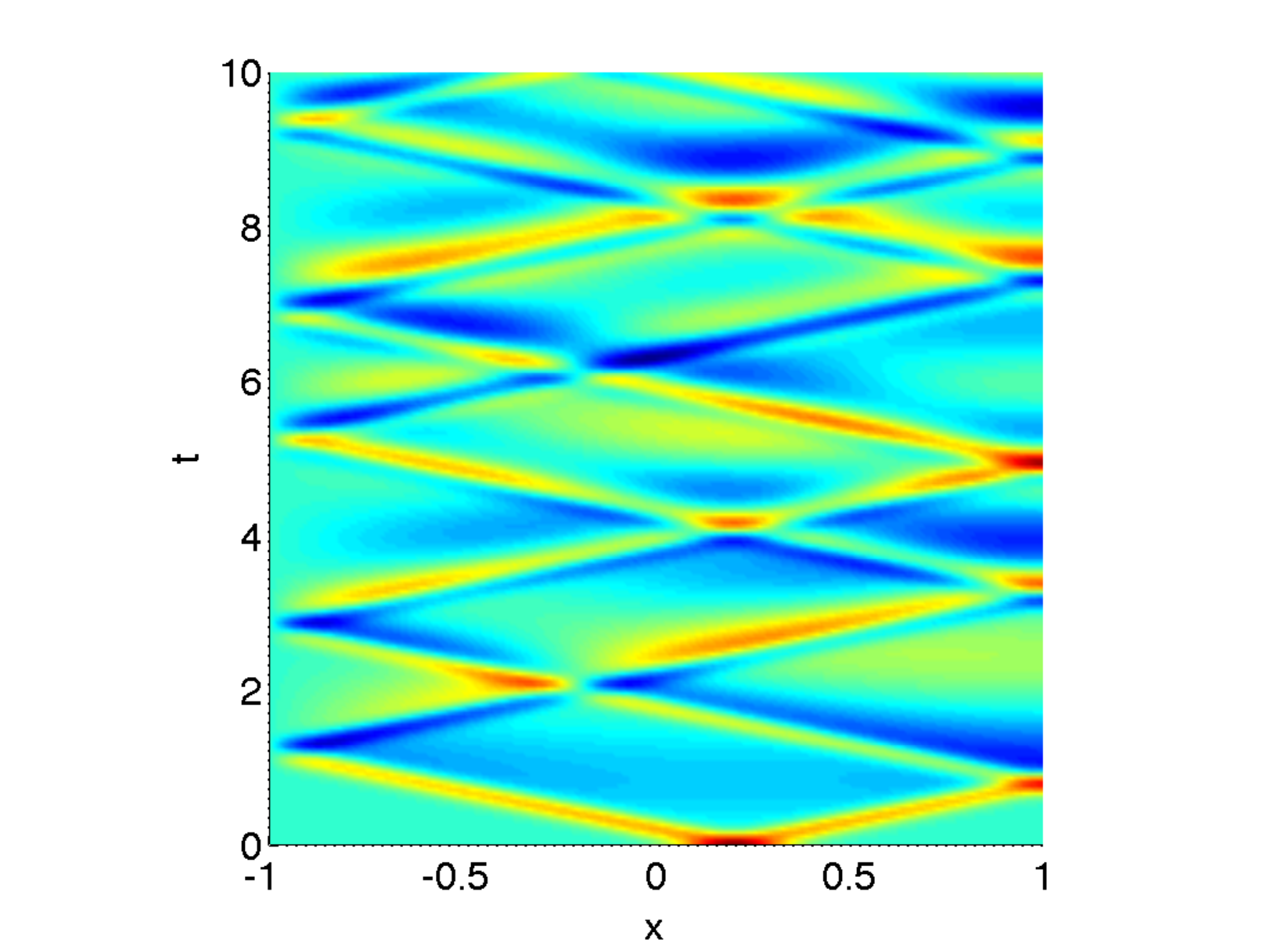}
\end{overpic}

\end{minipage}
\caption{Left: The solution to the wave equation 
$u_{tt}=u_{xx}$ with $u(x,0) = e^{-50(x-2/10)^2}$, $u_t(x,0)=0$, $u(-1,t) =0$, 
and $u(1,t) + 5u_x(1,t) =0$. Right: The solution to the 
Klein--Gordon equation $u_{tt}=u_{xx} - 5u$ with the same boundary conditions 
as for the wave equation. }
\label{fig:WaveGordon} \end{figure}

\subsection*{Example~4: The time-dependent Schr\"odinger equation}
For the fourth example we consider the time-dependent Schr\"odinger equation on $[0,1] \times [0,0.54]$,
\begin{equation} 
\I \epsilon u_t = -\frac{1}{2}\epsilon^2 u_{xx} + V(x) u,
\label{eq:Schrodinger} 
\end{equation} 
with $u(0,t) = 0$, $u(1,t)=0$, and an initial condition $u(x,0) = u_0(x)$, where
\[
u_0(x) = e^{-25(x-1/2)^2} e^{-{\I / (5 \epsilon)} \log (2 \cosh ( 5(x-1/2)))}.
\]
In Figure~\ref{fig:BaoSchrodingerRe0256} we take $\epsilon = 0.0256$ and plot the real part of the solution when $V(x)=10$ (left) and $V(x) = x^2$ (right). In both cases, we see the formation of a caustic.  
\begin{figure} 
\centering 
\begin{minipage}{.49\textwidth} 
\centering 
\includegraphics[width=\textwidth]{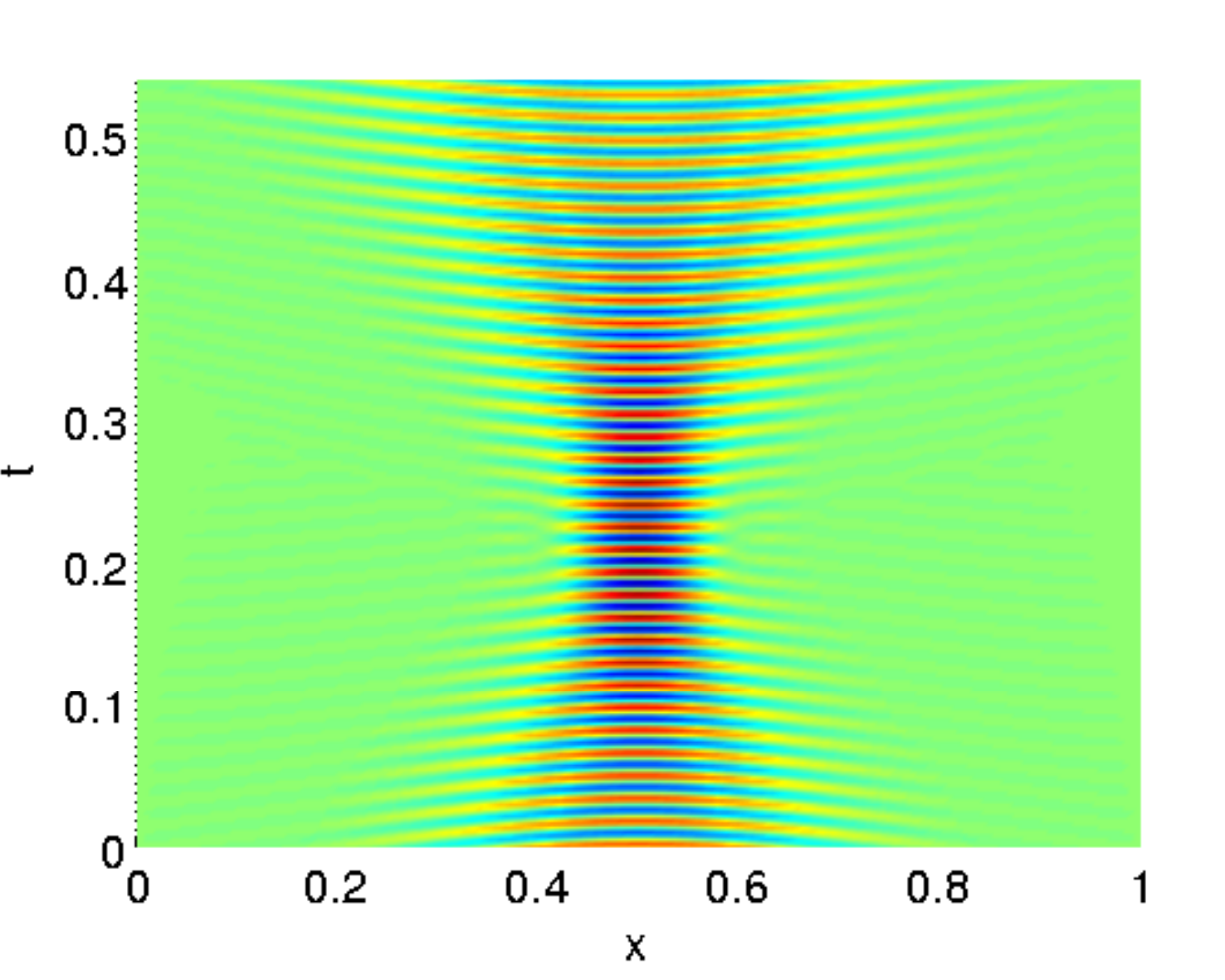}
\end{minipage}
\begin{minipage}{.49\textwidth} 
\centering 
\includegraphics[width=\textwidth]{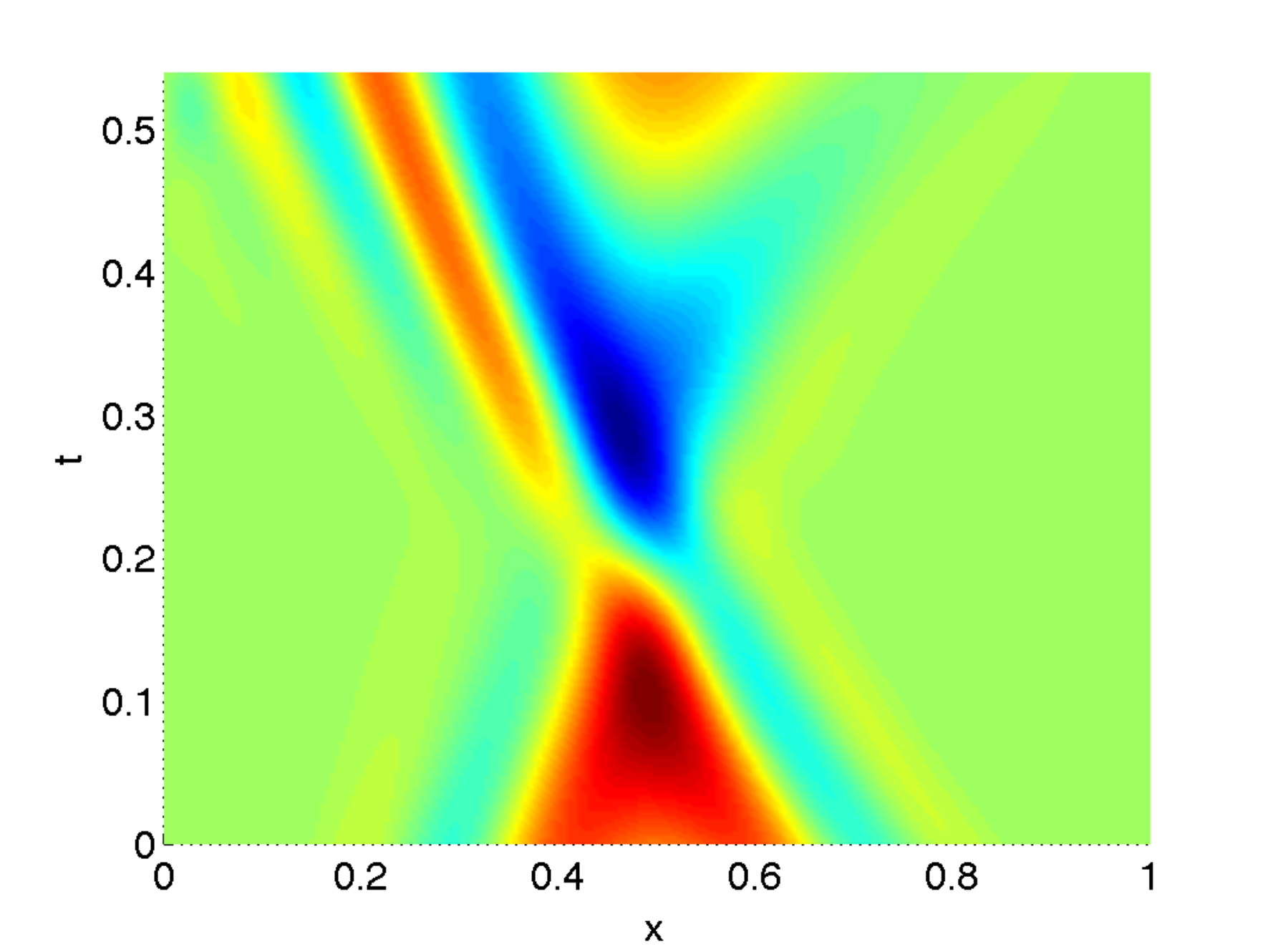}
\end{minipage}
\caption{The real part of the solution to~\eqref{eq:Schrodinger} 
with $\epsilon = 0.0256$ for $V(x) = 10$ (left) and $V(x) = x^2$ (right).}
\label{fig:BaoSchrodingerRe0256}
\end{figure}

In \figref{BaoSchrodinger0256} we plot $\abs{u(x,0.54)}^2$ where $u$ is the solution to~\eqref{eq:Schrodinger} with $V(x) = 10$  and $\epsilon = 0.0064$ (left). Our results are consistent with~\cite[Fig.~2b]{BaoSemiclassicalSchrodinger}, which used periodic boundary conditions in place of Dirichlet. In~\figref{BaoSchrodinger0256} (right) we plot the real and imaginary part of $\abs{u(x,0.54)}^2$ for $V(x) = x^2$ and $\epsilon = 0.0256$. We include this example as a demonstration of the versatility of our 2D spectral method and are not arguing that it is computationally competitive to custom built methods such as those in~\cite{BaoSemiclassicalSchrodinger}.

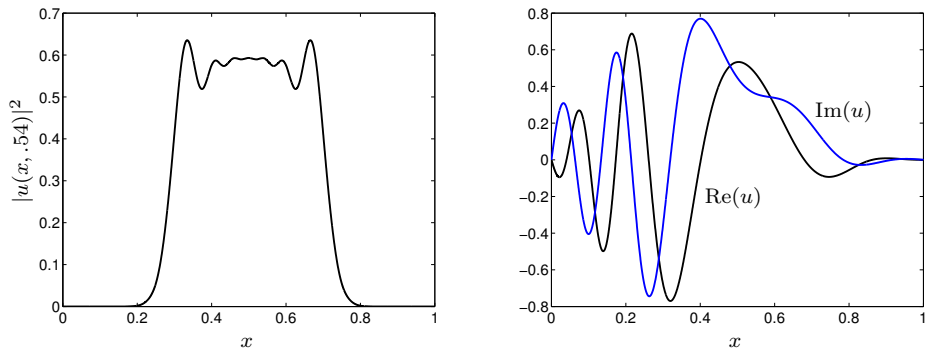
\begin{figure} 
\centering 
\begin{minipage}{.49\textwidth} 
\centering 
\begin{overpic}[width=\textwidth]{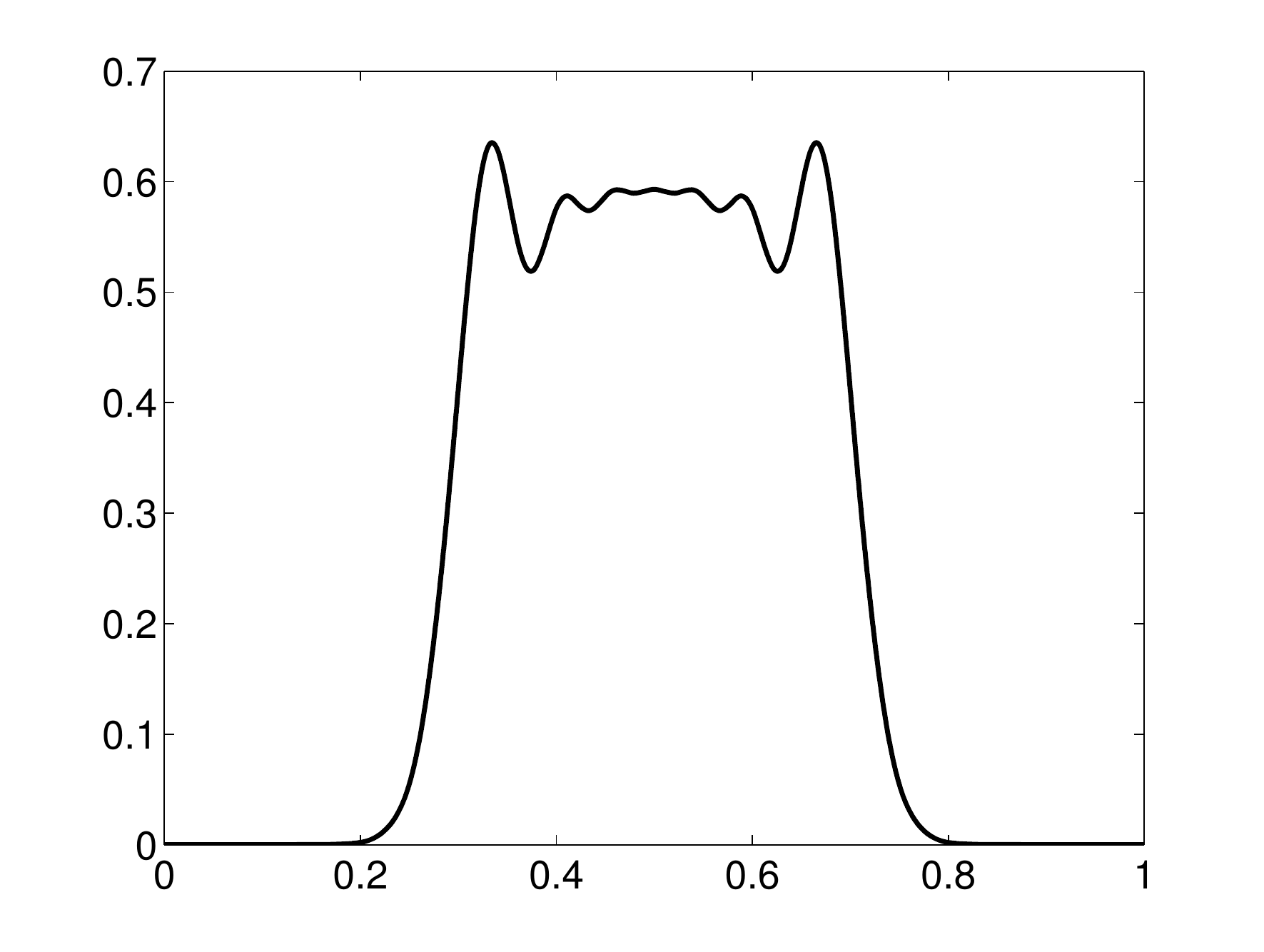}
\put(2,30){\footnotesize{\rotatebox{90}{$|u(x,.54)|^2$}}}
\put(50,0){\footnotesize{$x$}}
\end{overpic}
\end{minipage}
\begin{minipage}{.49\textwidth} 
\centering 
\begin{overpic}[width=\textwidth]{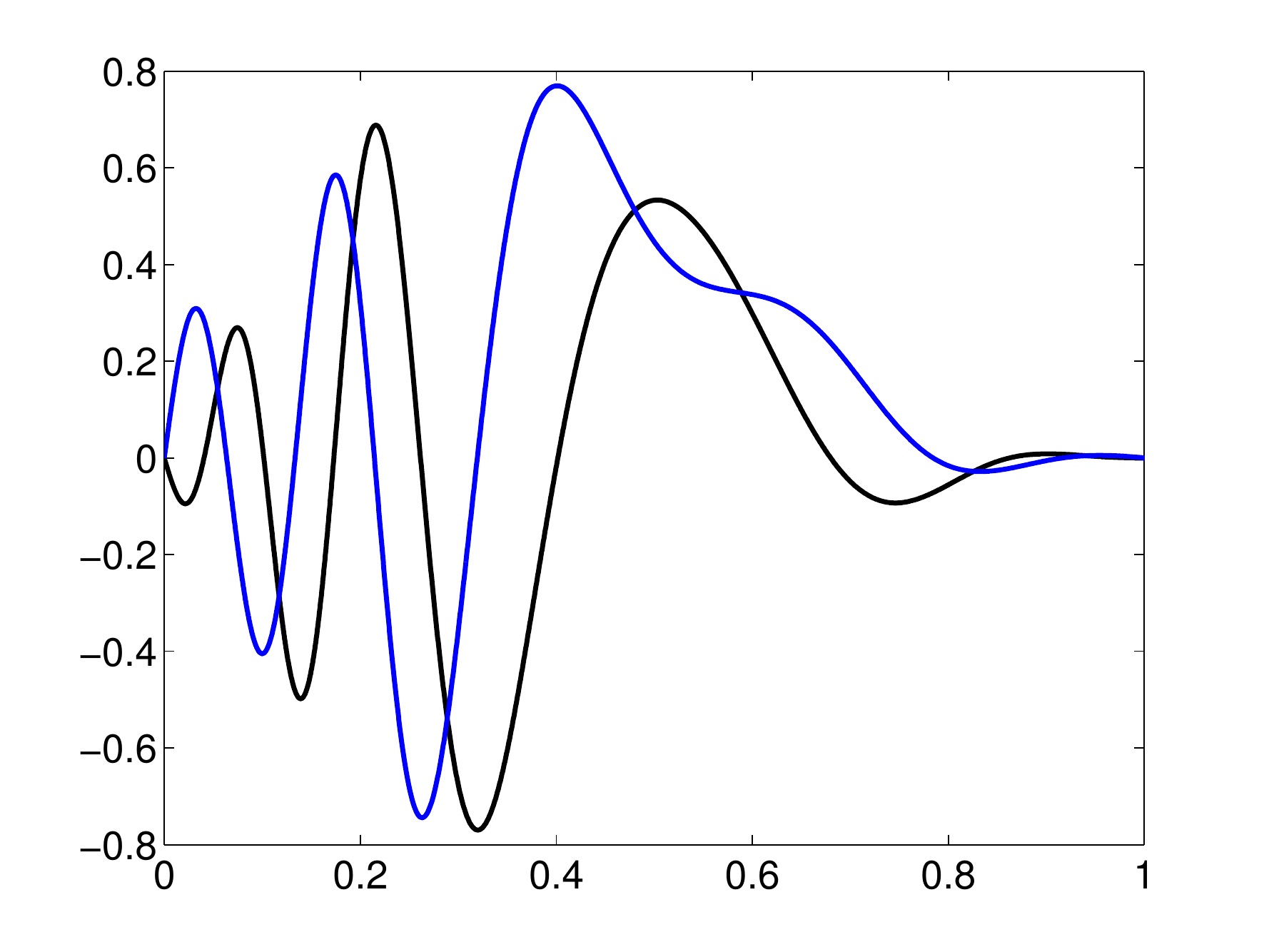}
\put(50,0){\footnotesize{$x$}}
\put(45,30){\footnotesize{${\rm Re}(u)$}}
\put(68,48){\footnotesize{${\rm Im}(u)$}}
\end{overpic}
\end{minipage}
\caption{Solution to the time-dependent Schr\"odinger equation at $t = 0.54$.  Left: The quantity $|u(x,0.54)|^2$ when $\epsilon = 0.0064$ and $V(x) = 10$. Right: The real (black) and imaginary (blue) part of $u(x,0.54)$ when $\epsilon = 0.0256$ and $V(x) = x^2$.}
\label{fig:BaoSchrodinger0256}
\end{figure}	

\subsection*{Example~5: The biharmonic equation}
The last example we consider is the biharmonic equation, a fourth order PDE, given by
\[
 u_{xxxx} + u_{yyyy} +2 u_{xxyy} = 0, \qquad (x,y)\in[-1,1]^2
\]
with Dirichlet and Neumann boundary data corresponding to the function 
\[
v(x,y) = {\rm Im}\left( (x-\I y) e^{-2(x+\I y)} + \cos(\cos(x+\I y))\right)
\]
so that the solution is $u=v$. The implementation adaptively finds that a bivariate Chebyshev expansion
of degree $(30,30)$ is sufficient to resolve the solution to a maximum absolute error of $4.59\times 10^{-13}$ taking $2.52$ seconds. 
This is a PDO with a splitting rank of~$3$ and hence the algorithm solves a large (but almost banded) linear system rather than a Sylvester matrix equation (see Section~\ref{sec:matrixsolver}). When the underlying 
PDO has a splitting rank of~$k\geq3$, the almost banded structure of ultraspherical spectral discretizations
allows for fast linear algebra and hence, an $\mathcal{O}(\min(n_x^{3} n_y,n_x n_y^{3}))$ complexity of the solver.

\section{Future work}\label{sec:future} 

The approach presented extends naturally to vector-valued PDOs.  As an example, consider the bivariate Stokes flow 
equation with zero Dirichlet conditions on $u$ and $v$:
	\meeq{
		{u(\pm1,y)=u(x,\pm1)}= 0,\qquad v(\pm 1,y)=v(x,\pm1) = 0,\ccr
		{\partial^2  u \over \partial x^2} + {\partial^2  u \over \partial y^2}  - {\partial p \over \partial x} = -f_x, \ccr
		{\partial^2  v \over \partial x^2} + {\partial^2  v \over \partial y^2}  - {\partial p \over \partial y} = -f_y \qquad \hbox{and} \ccr
		{\partial u \over \partial x} + {\partial v \over \partial y} = 0.
		}
We can represent this in as a system of generalized Sylvester matrix equations as follows: 
	\meeq{ 
				{\mathfrak{B} U = U \mathfrak{B}^\top} = 0,\qquad
				\mathfrak{B} V = V \mathfrak{B}^\top = 0, \ccr				
				{\cal D}_2 U {\cal S}_0^\top {\cal S}_1^\top + {\cal S}_1 {\cal S}_0 U {\cal D}_2^\top - {\cal S}_1 {\cal D}_1 P {\cal S}_0^\top {\cal S}_1^\top = -F_x, \ccr
				{\cal D}_2 V {\cal S}_0^\top {\cal S}_1^\top + {\cal S}_1 {\cal S}_0 V {\cal D}_2^\top - {\cal S}_1 {\cal S}_0 P {\cal D}_1^\top {\cal S}_1^\top = -F_y \qquad\hbox{and}\ccr				
				{\cal D}_1 U {\cal S}_0^\top +{\cal S}_0 V {\cal D}_1^\top  = 0				
		}
The boundary conditions can be used to remove the dependency of the equation on the first two rows and columns of $U$ and $V$.  
The resulting reduced matrix equation can then be solved by vectorizing the matrices $U$, $V$, and $P$ and constructing a block-wise version 
of the Kronecker product of the operators.  This results in a significant increase in the bandwidth of the resulting operators, which may mean that this
approach is not competitive without the use of iterative solvers.

The technique of automatic differentiation is far more powerful than we have described and can be extended to compute 
the Fr\'{e}chet derivatives of nonlinear partial differential equation, allowing one to ``linearize'' and apply Newton's method in function 
space.  For many years a similar approach has been employed in 1D to solve nonlinear ODEs in {\sc Chebop}~\cite{Birkisson_12_01}.  
Unfortunately, rank-2 linear PDOs will rarely arise after linearization and the bandwidth of the operators will be comparable to the 
discretization required.

The spectral method we have described does, with some extra complications, extend to domains that can be decomposed into rectangles such 
as L-shaped domains. Such domains can be dealt with by solving coupled generalized Sylvester matrix equations 
with extra constraints imposing continuity of the solution.  A significant challenge is to resolve potentially strong 
corner singularities in a solution that can result from intruding corners of the domain. 
General domains present a major challenge for global spectral methods.  

We have presented a fast direct solver for PDOs with a splitting rank of~$2$, requiring $\mathcal{O}((n_xn_y)^{3/2})$ 
operations to compute a degree $(n_x,n_y)$ bivariate polynomial approximation. However, for PDOs with a splitting rank of
$k\geq 3$ we constructed a large almost banded matrix and solved the resulting 
linear system (see Section~\ref{sec:matrixsolver}) in $\mathcal{O}(\min(n_x^3n_y,n_xn_y^3))$ operations. 
It would be interesting to investigate possible direct algorithms for solving generalized Sylvester matrix 
equations of the form~\eqref{eq:matrixequation} with $k\geq 3$ terms. 

\section*{Conclusion} 
We have described a spectral method for solving linear PDEs defined on rectangles. 
The first step was to extract the 
variable coefficients of a PDO from an anonymous operator using automatic differentiation. 
Then, by calculating a separable representation for the PDO we exploited the remarkable 
properties of the 1D ultraspherical spectral method to achieve a general, automated, and 
fast linear 2D PDE solver. The resulting 2D spectral method has a complexity of 
$\mathcal{O}(n_x^3 + n_y^3)$ for PDOs with a splitting rank of~$2$, when the solution is approximated by 
a bivariate polynomial of degree $n_x$ in $x$ and degree $n_y$ in $y$. The solver 
is part of {\sc Chebfun} and is able to accurately solve a wide range of variable coefficient PDEs.

\section*{Acknowledgments}
We would like to thank \'{A}sgeir Birkisson for suggesting automatic differentiation as a way to
extract out the variable coefficients of a PDO from an anonymous handle. 
We also thank Nick Trefethen for reading various drafts of 
this manuscript, and Hadrien Montanelli for carefully working through the material. 
The referees gave us excellent feedback that lead to an improvement in the paper. 
We acknowledge the support of the European Research Council under the European
Union's Seventh Framework Programme (FP7/2007-2013)/ERC grant agreement 291068 (AT) and 
the support of the Australian Research Council through the Discovery 
Early Career Research Award (SO).

\end{document}